\documentclass{amsart}
\usepackage{mathptmx}
\usepackage{xcolor}
\usepackage{hyperref}
\usepackage[utf8]{inputenc}
\usepackage{amsaddr}
\usepackage[all]{xy}
\usepackage{amssymb}
\usepackage{amsmath}
\usepackage{enumerate}
\usepackage[margin=2.5cm]{geometry}
\usepackage{mathrsfs}

\newtheorem{theorem}{Theorem}[section]

\newtheorem{proposition}[theorem]{Proposition}
\newtheorem{definition}{Definition}
\newtheorem{corollary}[theorem]{Corollary}
\newtheorem{lemma}{Lemma}
\newtheorem{remark}{Remark}

\newtheorem{ithm}{Theorem}[section]

\newtheorem{icor}[ithm]{Corollary}

\def\cal#1{\mathcal{#1}}
\def\bb#1{\mathbb{#1}}
\def\lie#1{\mathfrak{#1}}

\usepackage[all]{xy}

\def\co{\colon}

\newcommand{\h}{\frac{1}{2}}

\newcommand{\ga}{\textsl{g}}

\makeatletter
\renewcommand{\email}[2][]{%
	\ifx\emails\@empty\relax\else{\g@addto@macro\emails{,\space}}\fi%
	\@ifnotempty{#1}{\g@addto@macro\emails{\textrm{(#1)}\space}}%
	\g@addto@macro\emails{#2}%
}
\makeatother

\title{Gromoll--Meyer's actions and the geometry of (exotic) spacetimes}
\author{Leonardo F. Cavenaghi}
\address{Instituto de Matemática, Estatística e Computação Científica -- Unicamp, Rua Sérgio Buarque de Holanda, 651, 13083-859, Campinas, SP, Brazil}
\email{leonardofcavenaghi@gmail.com}

\author{Lino Grama}
\address{Instituto de Matemática, Estatística e Computação Científica -- Unicamp, Rua Sérgio Buarque de Holanda, 651, 13083-859, Campinas, SP, Brazil}
\email{lino@ime.unicamp.br}

\begin{document}

\begin{abstract}
	Since the advent of new pairwise non-diffeomorphic structures on smooth manifolds, it has been questioned whether two topologically identical manifolds could admit different geometries. Not surprisingly, physicists have wondered whether a smooth structure assumption different from some classical known models could produce different physical meanings. Motivated by the works \cite{cmp/1103943444, wittenrelated}, \cite{brans, Sladkowski1996}, in this paper, we inaugurate a very computational manner to produce physical models on classical and exotic spheres that can be built equivariantly, such as the classical Gromoll--Meyer exotic spheres. As first applications, we produce Lorentzian metrics on homeomorphic but not diffeomorphic manifolds that enjoy the same physical properties, such as geodesic completeness, positive Ricci curvature, and compatible time orientation. These constructions can be pulled back to higher models, such as exotic ten spheres bounding spin manifolds, to be approached in forthcoming papers.
\end{abstract}
\maketitle

\section{Introduction}
The notion of differentiation in a spacetime model is usually inherited from the setup considered: observe that whenever we fix a background to study differential geometry (and to physical content, especially in general relativity and field theory, the advantage is taken from this mathematical field), physicists have constantly assumed that smoothness is more a necessary tool to the computations than it contains a hidden physical meaning. However, due to the seminal work of John Milnor \cite{mi} leading to the discovery of exotic structures keeping the (point set) topology unchanged, it is natural to ask whether smoothness can obstruct global descriptions of spacetime models. 

The exotic $ 7$ spheres constructed by Milnor 1956 are all examples of fiber bundles over the $ 4$ sphere 
with structure group $\mathrm{SO}(4)$. By classifying bundles on spheres via the \emph{clutching construction}\footnote{The clutching construction recovers $G$-principal bundles from a cocycle in $G$-Cech cohomology. For sphere bundles, it is given by the covering of the two $n$-dimensional hemispheres that overlap at the equator and one single transition function on that equator ($\mathrm{S}^{n-1}\rightarrow G$).} These correspond to homotopy classes of maps $\mathrm{S}^3\rightarrow \mathrm{SO}(4)$; see \cite{mi,mi3, ms} for clarifications and additional results. In this paper, we describe explicitly three models for compact Lorentzian spacetimes, which are base of principal bundles with structure group $\mathrm{S}^1\subset \mathrm{SU}(2)$, endowed with Lorentzian metrics for which the total space has a semi-Riemannian submersion metric. Moreover, on the base manifolds, the Ricci curvature is shown to be strictly positive in all the cases despite the signature of the metric being mixed. 

The main motivation in our constructions comes from trying to understand if two non-compatible smooth structures on the same ``spacetime'' could enjoy different physical rules. In this direction, we remark on the works of \cite{cmp/1103943444, wittenrelated, brans, Sladkowski1996}. For the constructions, we base ourselves on the following (see \cite{SperancaCavenaghiPublished, CavenaghiSperanca+2022+95+104}): Consider a compact connected principal bundle $G\hookrightarrow P \to M$ with a principal action $\bullet.$ Assume that there is another action (which we denote by $\star$) on $P$ that commutes with $\bullet.$ This makes $P$ a $G\times G$--manifold. If one assumes that $\star$ is free, one gets a $\star$-diagram of bundles:
\begin{equation*}
\begin{xy}\xymatrix{& G\ar@{..}[d]^{\bullet} & \\ G\ar@{..}[r]^{\star} & P\ar[d]^{\pi}\ar[r]^{\pi'} &M'\\ &M&}\end{xy}
\end{equation*}
In the former diagram, $M$ is the quotient of $P$ by the $\bullet$-action, and $M'$ is the quotient of $P$ by the $\star$-action. Since $\bullet$ and $\star$ commute, $\bullet$ descends to an action on $M'$ and $\star$ descends to an action on $M$. We denote the orbit spaces of these actions by $M'/\bullet$ and $M/\star$, respectively. It is worth pointing recalling that such diagrams were first introduced in \cite{duran2001pointed} and later generalized in \cite{speranca2016pulling, SperancaCavenaghiPublished, CavenaghiSperanca+2022+95+104}.

In Section \ref{sec:minkowski}, we use Cheeger deformations to recover a construction for the classical Minkowski spacetime. Based on this, we pass to consider the manifold $\mathrm{SU}(2)$ with an \emph{almost semi-free} isometric action by $\mathrm{S}^1$ and build an (almost) Lorentzian metric on $\mathrm{SU}(2)$ with positive Ricci curvature. More precisely, we consider a $\mathrm{S}^1$ action on $\mathrm{SU}(2)$, which is semi-free outside some fixed points. After taking care of extending a Cheeger deformation of a fixed bi-invariant metric on $\mathrm{SU}(2)$, we construct a semi-Riemannian submersion
\[(\mathrm{SU}(2)\times \bb R,\mathrm{h}_{-r^2}) \rightarrow (\mathrm{SU}(2),\ga_{-r^2})\]
where $\ga_{-r^2}, \mathrm{h}_{-r^2}$ are Lorentzian metrics. An interesting aspect of this construction is that the classical \emph{de Sitter spacetime} has the same topology as $\mathrm{SU}(2)\times \bb R$, but neither $\ga_{-r^2}$ nor $\mathrm{h}_{-r^2}$ are Einstein metrics.

Given the physical interest on $7$-dimensional spheres (classical and exotic), for instance, as described in \cite{cmp/1103943444}, we built explicit Lorentzian metrics on $\mathrm{S}^7$ and $\Sigma_{GM}^7$ (the Gromoll--Meyer exotic sphere (\cite{gromoll1974exotic})) which have equivalent (in some sense), physical properties related to the Lorentzian metrics there built. 
We list our results:
\begin{ithm}
    Let $M'\leftarrow P \rightarrow M$ be a star diagram with $\mathrm{S}^1$ as its structure group. Then there is a Lorentzian metric $\ga_{-r^2}$ in the principal stratum of $M^{reg}$ of the $\star$-action $M$, and a $\ga'_{-r^2}$ Lorentzian metric in the principal stratum ${M'}^{reg}$ of the $\bullet$-action on $M'$ such that the orbit spaces $M/\mathrm{S}^1$ and $M'/\mathrm{S}^1$ are isometric. Moreover, these metrics can be smoothly extended to Lorentzian metrics on $M, M'$ with positive Ricci curvature.
\end{ithm}

The following description of vectors guarantees the extension of the metric in the former theorem in the fixed points.

  \begin{ithm}
      Let $(M,\ga)$ be a path-connected semi-Riemannian manifold with an effective isometric action by $\mathrm{S}^1$ and $x$ be a $\mathrm{S}^1$-fixed point. Assuming that $\ga_{-r^2}$ is a Lorentzian metric in the principal stratum of $M$ induced by a $-r^{2}$-Cheeger deformation of $\ga$, then given $X\in T_xM$, the geodesic generated by $s\mapsto \gamma(s)$ is time-like for any $s>0$ if, and only if, $\lie g_X = \lie{s}^1$, where $\lie g_X$ is the Lie algebra of the isotropy subgroup at $X$ of the isotropy representation in $T_xM$.
  \end{ithm}

  \begin{icor}
      The classical sphere $\mathrm{S}^7$ and the Gromoll--Meyer exotic sphere $\Sigma_{GM}^7$ admit time-oriented Lorentzian metrics of positive Ricci curvature with isometric semi-free actions of $\mathrm{S}^1$ on an open and dense subset (regular stratum). Moreover, these actions fix some points out of the regular stratum, and the orbit-spaces $\mathrm{S}^7/\mathrm{S}^1$ and $\Sigma_{GM}^7/\mathrm{S}^1$ are isometric as metric spaces.
  \end{icor}
   A possible conclusion for the former corollary is that Einstein's equations recover classical and exotic differential topology on spacetimes. We want to thank Prof. Ernesto Lupercio for pointing it out.

  A further result is proved to ensure completeness of the built metrics on exotic and classical spheres:
  \begin{ithm}
        Let $M'\leftarrow P \rightarrow M$ be a star diagram with $\mathrm{S}^1$ as its structure group. Then there is a Lorentzian metric $\widetilde \ga_{-r^2}$ in $M$, and a Lorentzian metric ${\widetilde\ga}'_{-r^2}$ in $M'$, with complete space and time-like geodesics.
  \end{ithm}

\section{Fixing the notation and a minimum of Lorentzian geometry}

Every aspect of this section is deeply based in \cite{victoria2010introduction, oneillbook}.

\subsection{Recordatory aspects of Linear algebra of indefinite metrics}
 A semi-Euclidean space $V$ is a finite-dimensional real $n$-dimensional vector space with a \emph{non-degenerate} quadratic form $q$. Such a quadratic form can, via a choice of basis $\cal B:= \{e_1,\ldots,e_n\}$, be computed in any vector $v\in V$ furnishing
 \[q(v) = v_1^2 + \ldots v_p^2 - (v_{p+1}^2 + \ldots v_{p+q}^2).\]

 The quadratic form $q$ gives rise to a symmetric bilinear $\ga$ form via the classical polarization formula. Given $v,w \in V$ then $v\perp w$ (i.e., $v$ and $w$ are orthogonal) if $\ga(v,w) = 0$. Two vector subspaces $A, B \subset V$ are said to be orthogonal, and we denote $A\perp B$ if $v\perp w \forall v\in A,~w\in B$. Fixed any vector subspaces $A \subset V$ we define $A^{\perp}
:= \{w\in V : \ga(v,w) = 0, \forall v\in A\}.$

\begin{definition}
Let $g$ be a symmetric bilinear form. Given $v\in V$, $q_{\ga}(v) := \ga(v,v)$, we say that $v$ is
\begin{enumerate}[(i)]
    \item time-like if $q_{\ga}(v) < 0$,
\item light-like if $q_{\ga}(v) = 0$ and $v\neq 0$,
\item space-like if $q_{\ga}(v) > 0$,
\item causal if $v$ is time-like or light-like.
\end{enumerate}
\end{definition}

\begin{definition}
    A scalar product $\ga$ on $V$ is a non-degenerate symmetric bilinear form.
\end{definition}

Let $\{e_1,\ldots,e_n\}$ be a basis of $V$. We say that this basis is \emph{orthonormal} if $|e_i| := \sqrt{|\ga(e_i,e_i)|} = 1$ and $\ga(e_i,e_j) = 0$ if $i\neq j$. A crucial result for the linear theory of symmetric bilinear forms is based upon the following:
\begin{proposition}
    The number $\nu$ of time-like vectors in a basis $\cal B$ of $(V,\ga)$ does not
depend on the basis but only on $(V,\ga)$, and it is called the \emph{index} of $(V,\ga)$. Moreover, $(V,\ga)$ admits an orthonormal basis.
\end{proposition}

\begin{definition}
    A scalar product $\ga$ in a real vector space $V$ is
    \begin{enumerate}[(i)]
        \item Euclidean if $\nu = 0$,
        \item Lorentzian if $\nu = 1$ and $n\geq 2$,
        \item It is indefinite if it is in a symmetric bilinear form.
    \end{enumerate}
\end{definition}

\begin{proposition}
    The subset of the time-like vectors (resp., causal; light-like if $n>2$) has two connected components.
Each one of these parts will be called a time-like cone (resp. causal
cone; light-like cone). Moreover, two time-like vectors $v$ and $w$ lie in the same
time-like cone if, and only if, $\ga(v,w) < 0$. In particular, each time-like cone is convex.
\end{proposition}

\begin{definition}
Let $(V,\ga)$ be a Lorentzian vector space. We will say that a
subspace of $W\subset V$ is
\begin{enumerate}
    \item space-like, if $\ga|_{W}$ is Euclidean,
    \item time-like, if $\ga|_{W}$ is non-degenerated with index $1$ (that is,
it is Lorentzian whenever $\dim W\geq 2$),
\item light-like if $g|_{W}$ is degenerated
\end{enumerate}
\end{definition}

The following proposition plays an essential (implicit though) role in this paper:
\begin{proposition}
    A subspace $W\subset V$ is time-like if, and only if, $W^{\perp}$ is space-like.
\end{proposition}

\subsection{The basics of semi-Riemannian Geometry}

A connected semi-Riemannian manifold $(M,\ga)$ is a differentiable $n$-dimensional manifold $M$ equipped with an everywhere non-degenerate, smooth, symmetric metric tensor $\ga$ of constant index $\nu \in \{1,\ldots,n\}$. Such a metric is called a semi-Riemannian metric. The pair $(M,\ga)$ will be called \emph{Riemannian} manifold if $\nu = 0$ and \emph{Lorentzian} if $\nu = 1$.

\subsubsection{Temporal orientability}

\begin{definition}
    A \emph{time orientation} of a Lorentzian vector space is a choice of one of the two time-like cones (or, equivalently, of one of the causal or light-like cones). The chosen cone will be called \emph{future}, and the other one, \emph{past}.
\end{definition}

Let $(M,\ga)$ be a Lorentzian manifold. Any tangent space $(T_xM,\ga_x),~x\in M$ is a Lorentzian vector space with two time-like cones, and one can choose one of them as a time orientation. The question of how this choice can be carried out continuously is the equivalent (in affirmative) answer to define a \emph{temporal orientation} on $M$. More precisely,
\begin{definition}
    A \emph{time orientation} in a Lorentzian manifold $(M,\ga)$ is a map which assigns, to each $x\in M$, a time-like cone $\tau_x \subset T_xM$ such that there exists an open neighborhood $U\ni x$ and a time-like vector field $X$ in $U$ such that $X_y \in \tau_y~\forall y\in U$. Any Lorentzian manifold assuming a time orientation is said \emph{time orientable}.
\end{definition}

It can be proved that
\begin{theorem}
    Let $M$ be a connected differentiable manifold. The following conditions are equivalent:
    \begin{enumerate}[(i)]
        \item $M$ admits a Lorentzian metric;
        \item $M$ admit a time-orientable Lorentzian metric;
        \item $M$ admits a vector field $X$ without zeros;
        \item Either $M$ is non-compact or its Euler characteristic is $0$.
    \end{enumerate}
\end{theorem}
As we shall see, every example in this paper admits a Lorentzian metric. Moreover, the built metrics shall be time-oriented.

\subsubsection{Computational aspects}
Let $(M,\ga)$ be a Lorentzian manifold. Then, just as in the Riemannian case, it is possible to consider the Levi-Civita connection $\nabla$ of $\ga$. Moreover, for such a connection is well defined, the Riemann curvature tensor $R_{\ga}(X, Y)Z:= \nabla_X\nabla_YZ - \nabla_Y\nabla_XZ - \nabla_{[X, Y]}Z$ for every vector fields $X, Y, Z$ in $M$. It also enjoys the Riemann curvature tensor symmetry and anti-symmetry property of the Riemannian case.

From the Riemannian curvature tensor $R_{\ga}$ can be derived the Ricci tensor (that is well defined even for light-like vectors) via
\[\mathrm{Ric}(X,Y)(x) := \mathrm{tr}\left(\mathrm{R}(\cdot,X)Y\right),\]
and the scalar curvature $\mathrm{scal}_{\ga} := \mathrm{tr}\mathrm{Ric}$.

However, the sectional curvature of $\ga$ is only defined equally as in the Riemannian case for a plane $\sigma \in T_xM,~x\in M$ if, and only if, $\sigma$ is non-degenerated: It has a positive or negative area. If $\sigma = \mathrm{span}_{\bb R}\{X,Y\}$ then $|X|^2|Y|^2 -\ga(X,Y)^2 \neq 0$.

\subsection{General Notation}

	We denote by $R_{\ga}$ the Riemannian tensor of the metric $\ga$: 
	\[R_{\ga}(X,Y)Z=\nabla_X\nabla_YZ-\nabla_Y\nabla_XZ-\nabla_{[X,Y]}Z,\]
	where $\nabla$ stands for the Levi-Civita connection of $\ga$. We denote either by $K_{\ga}(X, Y)=\ga(R_{\ga}(X, Y)Y, X)$ or by $R_{\ga}(X, Y) = R_{\ga}(X, Y, Y, X),$ making it clear in the context, the unreduced sectional curvature of $\ga$. We adopt the standard decomposition notation $TM = \cal V\oplus \cal H$ for any foliated manifold $M$, being $\cal V$ the subbundle collecting tangent spaces to the leaves of the foliation $\cal F$ and $\cal H$ to some subbundle complementary choice in $TM$.
	
	The superscripts $\mathbf{v}$ or $\mathbf{h}$ denote the projection of the underlined quantities on such subbundles, $\cal V$ and $\cal H$, respectively. Whenever we say we have a \emph{Riemannian principal bundle}, we mean that the principal bundle is considered with a Riemannian submersion metric.

\section{A Cheeger approach to the Minkowski space and a first compact model}
\label{sec:Cheeger}
We first recall the aspects of Cheeger deformations needed to motivate our constructions well. Though the main formulae come from classical references, such as \cite{mutterz} and \cite{Muter}, we follow \cite{Cavenaghi2022} for the presentation. The reason for that is to make the constructions of what we call ``compact models of exotic spacetimes'' more natural once the concept of $\star$-diagrams is established.
	
	Take the product manifold $M\times G$ with the product metric $\ga+t^{-1}Q$, where $G$ acts on $M$ via isometries and $Q$ is a bi-invariant metric on $G$. We, therefore, see ourselves with two possibilities of free (and commuting) actions:
	
	\begin{equation}\label{eq:modelo}
\begin{xy}\xymatrix{& G\ar@{..}[d]^{\bullet} & \\ G\ar@{..}[r]^{\star} & M\times G\ar[d]^{\pi}\ar[r]^{\pi'} &M\\ &M&}\end{xy}
\end{equation}

In \eqref{eq:modelo} the action $\bullet$ stands to
\begin{equation}\label{eq:bullet}
r\bullet (m,g) := (m,rg),
\end{equation}
while the action $\star$ is nothing but the associated bundle action on $M\times G$, that is
\begin{equation}\label{eq:star}
r\star (m,g) := (rm,rg).
\end{equation}

Therefore, $\pi((m,g)) := m$ meanwhile $\pi'((m,g)) := g^{-1}m$. Since $\pi$ and $\pi'$ define principal bundles, the metric $\textsl g+t^{-1}Q$ induces via $\pi$ and $\pi'$, respectively, the metrics $\ga$ (the original one) and $\textsl g_t$, a $t$-\emph{Cheeger deformation} of $\textsl g$. 
	
	Throughout the paper, it shall be denoted by $\mathfrak{m}_p$, the $Q$-orthogonal complement of $\mathfrak{g}_p$, the Lie algebra of the isotropy subgroup at $G_p$. We recall that $\mathfrak{m}_p$ is isomorphic to the tangent space to the orbit $Gp$ via \textit{action fields}: For any $U\in \lie g$ the corresponding action field out of $U$ is defined by the rule
	\begin{equation*}
	U^*_p=\frac{d}{dt}\Big|_{t=0}\mathrm{Exp}\left(tU\right)p
	\end{equation*}
 where $\mathrm{Exp} : \lie g \rightarrow G$ is the Lie group exponential map.
	It is straightforward to check that the map $U\mapsto U^*_p$ is a linear morphism whose kernel is $\lie g_p$. In this manner, any vector tangent to $T_pGp$ is said to be \emph{vertical}. Hence, such a space is named the \emph{vertical space} at $p$, denoted by $\mathcal{V}_p$. For each $p\in M$, its orthogonal complement, denoted by $\cal H_p$, is named \emph{horizontal space}. A tangent vector $\overline X \in T_pM$ can be uniquely decomposed as $\overline X = X + U^{\ast}_p$, where $X$ is horizontal and $U\in \lie m_p$.
	
	There are symmetric positive definite (whenever $t\geq 0$) tensors relating original metric quantities with deformed ones:
	\begin{enumerate}
		\item The \emph{orbit tensor} at $p$ is the linear map $P : \mathfrak{m}_p \to \mathfrak{m}_p$ defined by
		\[\textsl g(U^{\ast},V^{\ast}) = Q(PU,V),\quad\forall U^{\ast}, V^{\ast} \in \mathcal{V}_p\]
		\item For each $t > 0$ we define $P_t:\lie m_p\to \lie m_p$ as 
		\[\textsl g_t(U^{\ast},V^{\ast}) = Q(P_tU,V), \quad\forall U^{\ast}, V^{\ast} \in \mathcal{V}_p\]
		\item The \emph{metric tensor of $\textsl g_t$}, $C_t:T_pM\to T_pM$ is defined as
		\[\textsl g_t(\overline{X},\overline{Y}) = \textsl g(C_t\overline{X},\overline{Y}), \quad\forall \overline{X}, \overline{Y} \in T_pM\]
	\end{enumerate}
	The former tensors are related as
	\begin{proposition}[Proposition 1.1 in \cite{mutterz}]\label{propauxiliar}
		\begin{enumerate}
			\item \label{eq:pt} $P_t = (P^{-1} + t1)^{-1} = P(1 + tP)^{-1}$,
			\item If $\overline{X} = X + U^{\ast}$ then $C_t(\overline{X}) = X + ((1 + tP)^{-1}U)^{\ast}$.
		\end{enumerate}
	\end{proposition}
	
 In the field of non-negative/positive curvatures, the metric tensor $C_t^{-1}$ can be used to define a reparametrization of $2$-planes, not only making it easier to compute sectional curvature of a Cheeger deformed metric but allowing us to observe that Cheeger deformations (for non-negative $t$) do not create `new' planes with zero sectional curvature, meaning that
	\begin{theorem}\label{thm:curvaturasec}
		Let $\overline{X} = X + U^{\ast},~ \overline{Y} = Y + V^{\ast}$ be tangent vectors. Then \[\kappa_t(\overline X,\overline Y) := R_{\textsl g_t}(C_t^{-1}\overline{X},C_t^{-1}\overline{Y},C_t^{-1}\overline{Y},C_t^{-1}\overline{X})\] satisfies
		\begin{equation}\label{eq:curvaturaseccional}
		\kappa_t(\overline X,\overline{Y}) = \kappa_0(\overline{X},\overline{Y}) +\frac{t^3}{4}\|[PU,PV]\|_Q^2 + z_t(\overline{X},\overline{Y}),
		\end{equation}
		where $z_t$ is non-negative if $t\geq 0$. Moreover, for points in regular orbits, we have
   \[ z_t(\overline X,\overline Y) = 3t\left\|(1+tP)^{-1/2}P\nabla^{\mathbf{v}}_{\overline X}\overline Y - (1+tP)^{-1/2}t\h[PU,PV]\right\|_Q^2.\]
	\end{theorem}
	We refer to either \cite[Proposition 1.3]{mutterz} or \cite[Lemma 3.5]{cavenaghiesilva} for the details on the proof and more references.

 We will use Cheeger deformations with group actions by $\bb R$ or $\mathrm{S}^1$ to produce Lorentzian metrics on some manifolds. As a first step, we show that such a procedure allows us to recover the classic \emph{Minkowski spacetime}. Later, we proceed to study what we see as an ``analogous compact model'' to the Minkowski spacetime since the same tool and the practical computations using quaternionic algebra allow us to recover the classical \emph{de Sitter spacetime} as some universal covering.

 Before proceeding, however, we make some remarks.
 \begin{remark}\label{rem:importante}
\begin{enumerate}[(a)]
    \item if the $G$-action on $M$ is free and effective and the orbits of $G$ are totally geodesic for $\ga$, then the orbit tensor $P$ is constant along the whole manifold once seen as a section of the bundle (over $M$) of endomorphisms of $\lie g$. In this manner, it is convenient sometimes to take $P = 1$ (the identity operator) under the assumption of totally geodesic leaves.
    \item for effective free $G$-actions, we can always assume that the orbits of $G$ are totally geodesic for some Riemannian metric $\ga$, which preserves the initial horizontal distribution. This is due to Searle--Solorzano--Wilhelm, \cite{solorzano}. We often use this implicitly, combined with the previous item.
    \item the metric tensor $C_t = (1+tP)^{-1}$ contains the needed control to produce (semi)-Riemannian metrics via Cheeger deformations. Observe that, as operators, if one takes $t < 0$, then $\ga_{t}$ has mixed signature if, and only if, $1<-tP$. For most of this paper, $P=1$, and we choose $t = -r^{2}$ assuming that $r^2 > 1$, aiming to produce Lorentzian metrics.
\end{enumerate}
 \end{remark}

\subsection{The standard Minkowski spacetime}
\label{sec:minkowski}
The description of the so-called \emph{Minkowski spacetime} $\bb{R}^{1,3}$ is obtained from the topological space $\bb R^4$, providing it a semi-Riemannian metric.

Let $\bb H$ be the quaternionic space with $\mathrm{dim}_{\bb R}\bb H = 4$. In what follows, we shall use the inherited algebraic structure of $\bb H$ to perform all the $\bb R^4$ computations. In this manner, throughout this section, we identify
\[\bb R^4 \cong \bb H\]
\[\vec{q} = (x^0,x^1,x^2,x^3) \leftrightarrow x^0 + x^1\mathbf{i} + x^2\mathbf{j} + x^3\mathbf{k} = q\]

Let 
\[ t\cdot \vec{q} := (t+x^0) + x^1\mathbf{i} + x^2\mathbf{j} + x^3\mathbf{k}.\]
be an effective free action of $\bb R \cong \mathrm{Re}(\bb H)$ on $\bb H$ and consider the product $\bb H \times \bb R$ with the following \emph{twisted} $\bb R$-action:
\begin{equation}\label{eq:twisted}
t\star (\vec{q},t') := (\vec{q}\cdot t,t+t').
\end{equation}
Observe that the action \eqref{eq:twisted} works as the analogous $\star$-action defined in equation \eqref{eq:star}. In this manner, to introduce a family of Lorentzian metrics on $\bb H$, we take $r > 0$ and consider the semi-Riemannian metric
\begin{equation}
    \widetilde{\ga}_{-r^{2}} := \langle\cdot,\cdot\rangle_{\bb H} -r^{-2}d\theta^2
\end{equation}
in $\bb H\times \mathbb{R}$, where $\langle\cdot,\cdot\rangle_{\bb H}$ is the quaternionic inner product defined by
\[\langle \vec{q},\vec{r}\rangle_{\bb H} := \mathrm{Re}(\vec{q}\overline{\vec{r}}).\] Moreover, since we can see $\mathbb{R
}$ inside $\bb H$ as $\bb R \cong \mathrm{Re}~\bb H$ then $d\theta^2$ is nothing but standard multiplication of real numbers.

It is straightforward to check that the vertical (one-dimensional) space at $(q,t')$ coincides with
\begin{equation}\label{eq:vertmin}
\cal V_{(\vec{q},t')} := \mathrm{span}_{\bb R}\{(1 + 0\mathbf i + 0\mathbf j + 0\mathbf k,1)\}.
\end{equation}

We proceed to compute the horizontal space that is $\widetilde \ga_{-r^{2}}$ complementary to the vertical space \eqref{eq:vertmin}. With this aim, we must take $\vec{X} = (X,t') \in T_{(q,t')}\left(\bb H\oplus \bb R\right)\cong \bb H\oplus \bb R$ such that
\begin{equation*}
\langle X,1\rangle -r^{-2}t' = 0.
\end{equation*}

Since $\langle X,1\rangle = \mathrm{Re}(X\overline 1) = \mathrm{Re}(X)$ one gets that
$\mathrm{Re}(X) = r^{-2}t'$ so
\begin{equation}\label{eq:hormin}
\cal H_{(q,t')} = \left\{\vec{X} = (X,t') \in \bb H\times \{t'\} : \vec{X} = (r^{-2}t' + X^1\mathbf i + X^2\mathbf j + X^3\mathbf k,t')\right\}
\end{equation}
Furthermore, note that
\begin{align*}
|\vec{X}|^2_{-r^2} &= r^{-4}(t')^2 + |\mathrm{Im}~X|^2 -r^{-2}(t')^2\\
&= -(t')^2r^{-4}(r^{2}-1) + |\mathrm{Im}~X|^2,
\end{align*}
 so  for any $r > 1$ the metric $\ga_{-r^{2}}$ on $\bb H$ with
 \[\bb H \cong \mathbb{R}\oplus\mathbb{R}^3\cong \mathrm{span}_{\bb R}\left\{\partial_t\right\}\oplus \mathrm{Im}~\bb H\]
 \[\vec{q} = x^0 + x^1\mathbf i + x^2\mathbf j + x^3\mathbf k \leftrightarrow (\mathrm{Re}~\vec{q},\mathrm{Im}~\vec{q})\]
\begin{equation}
\ga_{-r^2}(t'\partial_t + X,t''\partial_t + Y) := -t't''r^{-4}(r^{2}-1) + \langle X,Y\rangle_{\bb R^3} 
\end{equation}
is a semi-Riemannian metric for which the vector field $\partial_t$, which is nothing but the identification $\frac{\partial}{\partial x^0} \equiv \partial_t$, is globally defined and time-like.

If one denotes $c^2 := r^{-4}(r^{2}-1)$ to make analogy with physics we have the Minkowski space $(\bb R^4,\ga_{-c^2})$ with the following description:
\begin{equation}
    \ga_{-c^2}(t'\partial_t + X,t''\partial_t + Y) := -t't''c^2 + \langle X,Y\rangle_{\bb R^3}.
\end{equation}

\subsection{$\mathrm{SU}(2)$ as a spacetime and obstructions}

This section sees $\mathrm{SU}(2)$ as the subset of $\bb H$ consisting of unit elements. It conveys to denote each element of $\mathrm{SU}(2)$ by $q\in \bb H$ such that $|q|^2 = 1$.

Although $\mathrm{S}^1$ is naturally included in $\mathrm{SU}(2)$, there are plenty of choices to explicitly write one of such to act on $\mathrm{SU}(2)$. We shall proceed by picking $\mathrm{S}^1 = \{e^{\mathbf{i}\theta} : \theta \in \bb R\}.$  To construct a $\mathrm{S}^1$-action on $\mathrm{SU}(2)$, instead of right or left multiplication, we look to the restriction of a $\mathrm{SU}(2)$-action to a $\mathrm{S}^1$-action on $\mathrm{S}^7$: such a construction comes out the $\star$-diagram below, appearing first in the work of Gromoll and Meyer \cite{gromoll1974exotic}, as well later in \cite{duran2001pointed,speranca2016pulling, SperancaCavenaghiPublished, Cavenaghi2022}.

		Consider the compact Lie group
		\begin{equation}
			\mathrm{Sp}(2) = \left\{\begin{pmatrix} a & c \\b & d\end{pmatrix}\in \mathrm{S}^7\times \mathrm{S}^7~ \Big| ~a\overline{b} + c\overline{d} = 0\right\},\label{eq:Sp2}
		\end{equation} 
		where $a,b,c,d\in \bb H$ are quaternions with their standard conjugation, multiplication, and norm.
		The projection $\pi:\mathrm{Sp}(2)\to \mathrm{S}^7$ of an element to its first row defines a principal $\mathrm{SU}(2)$-bundle with principal action:
		\begin{equation}\label{eq:GMprincipalaction}
			\begin{pmatrix} 
				a & c \\
				b & d 
			\end{pmatrix}\overline q = \begin{pmatrix}
				a & c\overline{q}\\
				b & d\overline{q}
			\end{pmatrix}.
		\end{equation}
		As in \cite{gromoll1974exotic}, we consider the $\star$-action
		\begin{equation}\label{eq:GMstaraction}
			q \begin{pmatrix} 
				a & c \\
				b & d 
			\end{pmatrix} = \begin{pmatrix} 
				qa\overline{q} & qc \\
				qb\overline{q} & qd 
			\end{pmatrix},
		\end{equation}
		whose quotient is an exotic 7-sphere, via the quotient map  
  \[\pi'\left( \begin{pmatrix} 
				a & c \\
				b & d 
			\end{pmatrix}\right) := \left(2\overline cd,|c|^2-|d|^2\right),\]
  placing everything in the diagram
		\begin{equation}\label{eq:CDGM}
			\begin{xy}\xymatrix{& \mathrm{SU}(2)\ar@{..}[d]^{\bullet} & \\ \mathrm{SU}(2)\ar@{..}[r]^{\star} & \mathrm{Sp}(2)\ar[d]^{\pi}\ar[r]^{\pi'} &\Sigma^7_{GM}\\ &\mathrm{S}^7&}\end{xy}
		\end{equation}
  (In \cite{Cavenaghi2022}, such construction is exploited to produce stair of bundles with curvature properties ).

  We consider \[\mathrm{SU}(2) = \left\{\begin{pmatrix}
  a\\
  0
  \end{pmatrix} \in \mathrm{S}^7\times \{0\}\subset \bb H\oplus \bb H\right\}\]
and define the following action of $\mathrm{S}^1$ on $\mathrm{SU}(2)$:
		\begin{equation}\label{eq:triviasu2}
q\cdot a := qa\overline q.
  \end{equation}

It should then be clear that
\begin{equation*}
\cal V_a \cong \mathrm{span}_{\bb R}\left\{\mathbf{i}a - a\mathbf{i} : a \in  \mathrm{SU}(2)\right\}.
\end{equation*}
Hence, if we now consider the $\mathrm{S}^1$-twisted action in $\mathrm{SU}(2)\times \mathrm{S}^1$:
\begin{equation}\label{eq:twistsu2}
q\ast (a,q') := (qa\overline q,qq'), ~q'\in \mathrm{S}^1, a \in \mathrm{SU}(2)
\end{equation}
one gets the following description of the vertical space for the action given by equation \eqref{eq:twistsu2}:
\begin{equation*}
\cal V_{(a,q')} \cong \mathrm{span}_{\bb R}\left\{(\mathbf{i}a - a\mathbf{i}, \mathbf{i}q')\right\}.
\end{equation*}
Observe that the action \eqref{eq:triviasu2} is nothing but the $\star$-action presented in equation \eqref{eq:star} used to perform a Cheeger deformation. However, to avoid confusion with equation \eqref{eq:GMstaraction}, we wrote it as $\ast$ instead of $\star$.

We regard $\mathrm{SU}(2)\times \mathrm{S}^1$ with the following $\mathrm{S}^1$-invariant semi-Riemannian metric
\[\widetilde{\ga}_{-r^{2}}(\cdot,\cdot) := \left\langle\cdot,\cdot\right\rangle_{\mathrm{SU}(2)} -r{^{-2}}d\theta^2,\]
where $d\theta^2$ stands for the standard metric on $\mathrm{S}^1$: observe that to our description of $\mathrm{S}^1$ we have that $d\theta^2(\mathbf{i}t,\mathbf{i}t')$ coincides with $\mathrm{Re}(\mathbf{i}t\overline{\mathbf{i}t'}) = tt'$, for $t,t'\in \bb R$. The metric $\langle\cdot,\cdot\rangle_{\mathrm{SU}(2)}$ is the isometric immersion metric obtained from $\langle\cdot,\cdot\rangle_{\bb H}$.

We then proceed to compute the horizontal space, complementary to $\cal V_{(a,q')}$ for $\widetilde{\ga}_{-r^2}$. To do so, let us search for a vector $(X\mathbf{i}a,-\mathbf{i}q't)$ with $X\mathbf{i}a \in T_a\mathrm{SU}(2)$ which satisfies
\[\widetilde{\ga}_{-r^{2}}((X\mathbf{i}a,-\mathbf{i}q't),(\mathbf{i}a-a\mathbf{i},\mathbf{i}q')) = 0\]
for some $t\in \bb R, ~X\mathbf{i}a\in \bb H$. First observe that $X\mathbf{i}a \in T_a\mathrm{SU}(2)$ if, and only if, $\mathrm{Re}(X\mathbf{i}a\bar a) = 0$. That is, $\mathrm{Re}(X\mathbf{i}) = 0$.

We compute:

\begin{align*}
0 &= \mathrm{Re}((\mathbf{i}a- a\mathbf{i})\overline{X\mathbf{i}a})  -r^{-2}\mathrm{Re}(\mathbf{i}q'\overline{(-\mathbf{i}q't)}) \\
&=\mathrm{Re}(\mathbf{i}a\overline a(-\mathbf{i})\overline X -a\mathbf{i}\overline a(-\mathbf{i})\overline X) - r^{-2}\mathrm{Re}(\mathbf{i}q'\overline{q'}\mathbf{i}t)\\
&= \mathrm{Re}(\mathbf{i}a\overline a(-\mathbf{i})\overline X -a\mathbf{i}\overline a(-\mathbf{i})\overline X) + r^{-2}t\\
&= \mathrm{Re}(\overline X+a\mathbf{i}\overline a\mathbf{i}\overline X) + r^{-2}t\\
&= \mathrm{Re}(\overline X-a\mathbf{i}\overline{\mathbf{i}a}\overline X) + r^{-2}t.
\end{align*}
Therefore,
\begin{equation*}
\cal H_{(a,q')} = \mathrm{span}_{\bb R}\left\{(X\mathbf{i}a,-\mathbf{i}q't) : \mathrm{Re}((1-a\mathbf{i}\overline{\mathbf{i}a})\overline X) = -r^{-2}t,~ \mathrm{Re}~X\mathbf{i} = 0\right\}.
\end{equation*}

Hence, to promote a semi-Riemannian metric on $\mathrm{SU}(2)$ via this procedure, one observes that since $|(X\mathbf{i}a,-\mathbf{i}q't)|_{-r^2}^2 = |X|^2_{\bb H} -r^{2}t^2$, via the identification $\mathbf{i} \leftrightarrow \partial_{\theta}$, we have:

\begin{equation}\label{eq:semiriemanniansu(2)}
T_a\mathrm{SU}(2) \cong\mathrm{span}_{\bb R}\left\{t\partial_{\theta}, X\mathbf{i}a : \mathrm{Re}((1-a\mathbf{i}\overline{\mathbf{i}a})\overline X) = -r^{-2}t,~\mathrm{Re}(X\mathbf{i}) = 0\right\}
\end{equation}
with the following semi-Riemannian metric
\begin{equation}\label{eq:semiriemanniansu(2)metric}
\ga_{-r^2}(t\partial_{\theta} + X\mathbf{i}a,t'\partial_{\theta} + Y\mathbf{i}a) = \langle X,Y\rangle_{\bb H} - tt'r^{-2}.
\end{equation}

However, observe that the defined $\mathrm{S}^1$-action on $\mathrm{SU}(2)$ has some fixed points. Therefore, for some points $a\in\mathrm{SU}(2)$, the time-like vector $\partial_{\theta}$ vanishes identically, and the metric degenerates to a Riemannian metric. Therefore, time does not exist for some points in this compact model, and the metric presents a kind of ``singularity''. The following section discusses the nature of this singularity, aiming to provide some physical insight into it.

\subsection{The de Sitter spacetime and the universal covering in a $\star$-diagram}

The \emph{de Sitter spacetime} corresponds to a solution to the \emph{Einstein field equations}
\[\mathrm{Ric}(\ga) -\h\mathrm{scal}_{\ga}\ga + \Lambda \ga = 0\]
that is topologically equivalent (in the homeomorphism sense) to $\bb R\times \mathrm{SU}(2)$, being a mathematical model of the universe that encompasses the observed accelerating expansion. We refer to \cite{Coxeter1943AGB} for a different geometric approach to the de Sitter spacetime.

Observe that the $(-r^2)$-Cheeger deformation obtained in the previous section allowed us to construct a semi-Riemannian submersion
\begin{equation}
    \bar \pi : \left(\mathrm{SU}(2)\times \mathrm{S}^1,\widetilde{\ga}_{-r^{2}}\right) \rightarrow \left(\mathrm{SU}(2),\ga_{-r^{2}}\right)
\end{equation}

On the other hand, the universal cover of $\mathrm{SU}(2)\times \mathrm{S}^1$ is homeomorphic to $\mathrm{SU}(2)\times \mathbb{R}$. In this manner, consider on $\mathrm{SU}(2)\times \mathbb{R}$ the covering semi-Riemannian metric $h_{-r^{2}}$ induced via the covering map
\[\mathrm{p} : \left(\mathrm{SU}(2)\times \bb{R},h_{-r^{2}}\right) \rightarrow  \left(\mathrm{SU}(2)\times \mathrm{S}^1,\widetilde{\ga}_{-r^{2}}\right).\]
We can build the commutative diagram of semi-Riemannian submersions
\begin{center}
    \begin{equation}\label{diagram:blowingupsu2}
        \xymatrix{
    \left(\mathrm{SU}(2)\times \mathrm{S}^1,\widetilde{\ga}_{-r^{2}}\right)   \ar[r]^{\bar \pi} &    \left(\mathrm{SU}(2),\ga_{-r^{2}}\right)  \\
    \left(\mathrm{SU}(2)\times \bb{R},h_{-r^{2}}\right) \ar[u]^{\mathrm{p}}  \ar[ru]_{\bar \pi \circ \mathrm{p}} }
    \end{equation}
\end{center}
which recovers the de Sitter spacetime $\left(\mathrm{SU}(2)\times \bb{R},h_{-r^{2}}\right)$, i.e., we have presented a (topologically) \emph{global} construction to the de Sitter spacetime (that is, \emph{coordinate free}) naturally associated with a Cheeger deformation.

Diagram \eqref{diagram:blowingupsu2} allow us to see, in some sense, $\left(\mathrm{SU}(2),\ga_{-r^{2}}\right)$ as a \emph{restrict world perspective of the de Sitter spacetime}. Said in other worlds, the submersion $\bar\pi \circ \mathrm{p}$ projects the vision of any observer in $\left(\mathrm{SU}(2)\times \mathbb{R},\mathrm{h}_{-r^{2}}\right)$ to the low-dimension compact model $\left(\mathrm{SU}(2),\mathrm{g}_{-r^{2}}\right)$. Moreover, observe that the $\mathrm{S}^1$-action on $\mathrm{SU}(2)\times \mathrm{S}^1$ is free, so there is no time-collapsing. In this sense, a way of justifying the singularity presented on the metric $\ga_{-r^2}$ relies on the $\mathrm{S}^1$-forced immersion on $\mathrm{SU}(2)$ as its orbits. The singularity can be understood as a ``lack of perception of the full spacetime''. 

We recall that a smooth group action $G$ on a smooth manifold $M$ is said to be semi-free if every isotropy group is discrete. It is said to be \textit{almost semi-free} if some fixed points exist other than discrete isotropy subgroups. Motivated by the former example, we define:
\begin{definition}\label{def:almostlorentzian}
    A Riemannian manifold $(M,\ga)$ with an effective isometric almost semi-free action by $\mathrm{S}^1$ is said to be an \emph{almost Lorentzian manifold}. The (almost) Lorentzian metric defined on the principal stratum $M^{reg}$ in $(M,\ga)$ is given by a sufficiently negative $-r^2$-Cheeger deformation of $\ga$.
\end{definition}
In Section \ref{sec:fixing}, we shall fix the metric on the fixed points so that we build a global Lorentzian metric on $\mathrm{SU}(2)$ as well as in any (almost) Lorentzian manifold.

In the next section, we take the previous description to obtain a higher dimension analogous to
\[\bar \pi \circ \mathrm{p} : \left(\mathrm{SU}(2)\times \mathbb{R},\mathrm{h}_{-r^{2}}\right)\rightarrow \left(\mathrm{SU}(2),\mathrm{g}_{-r^{2}}\right)\]
obtaining (almost) Lorentzian metrics in the classical seven-dimensional sphere $\mathrm{S}^7$ and in the Gromoll--Meyer exotic sphere: Both constructions coming from the diagram \eqref{eq:CDGM}. Before proceeding, however, we take advantage of equation \eqref{eq:curvaturaseccional} to check that Ricci curvature of $\left(\mathrm{SU}(2),g_{-r^{2}}\right)$ is positive but not constant. To prove this, we exploit the equation 
\begin{equation*}
    \kappa_{-r^{2}} = K_{\mathrm{SU}(2)} + z_{-r^{2}},
\end{equation*}
where $K_{\mathrm{SU}(2)}$ is the unreduced sectional curvature of the metric $\langle \cdot,\cdot\rangle_{\mathrm{SU}(2)}$ induced by $\langle\cdot,\cdot\rangle_{\bb H} := \mathrm{Re}(\cdot\overline{\cdot})$, which coincides with the negative of the Riemannian metric in $\mathrm{SU}(2)$ induced by a multiple of the Cartan--Killing form in $\lie{su}(2)$. Aiming generality on the computation, we will not use the explicit metric description furnished earlier but provide a general perspective using the formulae in Section \ref{sec:Cheeger}.

\begin{theorem}\label{thm:einsteintoy}
    For any $r > 1$ we have that $(\mathrm{SU}(2),\ga_{-r^2})$ is an almost Lorentzian manifold with time-like vectors given by
    \begin{equation*}
        \langle X,Y\rangle_{\bb H} - tt'r^{-2} < 0
    \end{equation*}
    and light-like vectors are given by
    \begin{equation*}
        \langle X,Y\rangle = tt'r^{-2}
    \end{equation*}
    where 
    \begin{align*}
        \mathrm{Re}((1-a\mathbf{i}\overline{\mathbf{i}a}\bar X) &= -r^{-2}t\\
        \mathrm{Re}((1-a\mathbf{i}\overline{\mathbf{i}a}\bar Y)) &= -r^{-2}t'\\
       \mathrm{Re}(X\mathbf{i}) = \mathrm{Re}(Y\mathbf{i})   &= 0
    \end{align*}
    $\forall~a\in \mathrm{SU}(2),~\forall X, Y \in \bb H,~t,t'\in \bb R$. Moreover, $\mathrm{Ric}_{\ga_{-r^{-2}}}$ is positive.
\end{theorem}

\begin{proof}
The first part of the statement was already obtained by our explicit description given by equations \eqref{eq:semiriemanniansu(2)} and \eqref{eq:semiriemanniansu(2)metric}. For the claim on the positivity of Ricci curvature, here we only compute the Ricci curvature in the regular part of the $\mathrm{S}^1$-action, that is, in the principal stratum. The analysis containing the behavior of Cheeger deformations for the fixed points shall be approached in Theorem \ref{thm:einsteingeneral}. Assuming that $\mathrm{S}^1$ is totally geodesic in $\mathrm{SU}(2)$ for the initial metric $\ga$, we can take $P = 1$ for the equation \eqref{eq:curvaturaseccional} in Theorem \ref{thm:curvaturasec}. Observe that such an assumption loses no generality due to Remark \ref{rem:importante}. Hence, we have that
\[z_{-r^{2}}(\overline X,\overline Y) = -3r^2(1-r^{2})^{-1}\left\|\nabla^{\mathbf{v}}_{\overline X}\overline Y\right\|_{\mathrm{S}^1}^2\]
and so
\[\mathrm{\kappa}_{-r^{2}}(\overline X,\overline Y) = \kappa_0(\overline X,\overline Y) -3r^2(1-r^{2})^{-1}\left\|\nabla^{\mathbf{v}}_{\overline X}\overline Y\right\|_{\mathrm{S}^1}^2.\]
Take $\{e_1,e_2\} \in \cal H_a$ a $\ga$-orthonormal basis to the horizontal space at $a\in\mathrm{SU}(2)$ and $\{e_3\} \in \mathbf{i}\bb R$ a $\ga$-unit vector. We shall add a $\sqrt{-1}$-multiplying factor on this element since we are considering $r^2-1>0$, so $(1-r^2)^{1/2}$ is a complex number. Then $\{\sqrt{-1}(1-r^2)^{1/2}e_3\}$ is a $\ga_{-r^{2}}$-unit vector in $\cal V_a$. Hence $\{e_1,e_2,\sqrt{-1}(1-r^2)^{1/2}e_3\}$ is a $\ga_{-r^{2}}$-orthonormal basis in $T_a\mathrm{SU}(2)$. 

Using the notation introduced in Section \ref{sec:Cheeger}, we denote as $\overline X = X + U^*$ any tangent vector to $(\mathrm{SU}(2),\ga)$. This manner 
\begin{align*}
\mathrm{Ric}_{\ga_{-r^{2}}}(\overline X) = \sum_{i=1}^2\mathrm{\kappa}_{-r^{2}}((1-r^2)^{-1}(\overline X),e_i) + \mathrm{\kappa}_{-r^{2}}((1-r^2)^{-1}(\overline X),\sqrt{-1}(1-r^2)^{1/2}e_3)\\
= \sum_{i=1}^2K_{\mathrm{SU}(2)}(e_i,X+ (1-r^2)^{-1}U^*) - (1-r^2)^{-1}K_{\mathrm{SU}(2)}(e_3,X+(1-r^2)^{-1}U^*) \\
-3r^2(1-r^{2})^{-1}\left\{\sum_{i=1}^2\left\|\nabla^{\mathbf{v}}_{e_i} X +(1-r^2)^{-1}U^*\right\|_{\mathrm{S}^1}^2 + \left\|\nabla^{\mathbf{v}}_{e_3} X +(1-r^2)^{-1}U^*\right\|_{\mathrm{S}^1}^2\right\}.
\end{align*}

To the following formulae, we recall that the O'Neill tensor is given by $A_XY:= \h[X, Y]^{\mathbf{v}}$ for any $X, Y\in \cal H$. The superscript $\mathbf v$ denotes 'the vertical component of.' We shall often consider $A^*_X$, the $\ga$-dual of $A_X$. Moreover, the shape operator $S_X$, i.e., the $\ga$-dual of each orbit's second fundamental $\sigma$, is the map $S_X: \cal V \rightarrow \cal V$. According to \cite[Theorem 1.5.1, p. 28]{gw}
\begin{align*}
    K_{\mathrm{SU}(2)}(e_i,X+(1-r^2)^{-1}U^*) &= K_{\mathrm{SU}(2)}(e_i,X) + (1-r^2)^{-2}K_{\mathrm{SU}(2)}(e_i,U^*)\\ + 2(1+r^2)^{-1}R_{\ga}(e_i,X,U,e_i) -3r^{2}(1-r^2)^{-1}|A_{X}e_i|_{\ga}^2 \\
    K_{\mathrm{SU}(2)}(e_3,X+(1-r^2)^{-1}U^*) &= K_{\mathrm{SU}(2)}(e_3,X) + 2(1+r^2)^{-1}R_{\ga}(e_3,X,U,e_3)\\
    \nabla^{\mathbf{v}}_{e_i}(X +(1-r^2)^{-1}U^*) &= A_{e_i}X - (1-r^2)^{-1}S_{e_i}U^*\\
    \nabla^{\mathbf{v}}_{e_3}(X +(1-r^2)^{-1}U^*) &= S_{X}(e_3)
\end{align*}
so the totally geodesic fiber assumption ensures that
\begin{multline*}
    \mathrm{Ric}_{\ga_{-r^{2}}}(\overline X) = \mathrm{Ric}^{\mathbf h}(X) + (1-r^2)^{-2}\sum_{i=1}^2|A^*_{e_i}U^*|_{\ga}^2 - (1-r^{2})^{-1}|A^*_Xe_3|_{\ga}^2 \\+2(1+r^2)^{-2}R_{\ga}(e_3,X,U,e_3) + 2\sum_{i=1}^2(1+r^2)^{-1}R_{\ga}(e_i,X,U^*,e_i) -3r^{2}(1-r^2)^{-1}\sum_{i=1}^2|A_Xe_i|_{\ga}^2.
\end{multline*}
Moreover, it can also be checked from the same formulae that
\begin{equation*}
  R_{\ga}(e_i,X,U^*,e_i) = \ga\left((\nabla^{\mathbf v}_iA)_iX,U^*\right)  
\end{equation*}
\begin{equation*}
    R_{\ga}(e_3,X,U,e_3) = 0.
\end{equation*}
\begin{equation*}
\ga\left((\nabla^{\mathbf v}_iA)_iX,U^*\right) = 0.
\end{equation*}
Hence,
\begin{multline*}
    \mathrm{Ric}_{\ga_{-r^{2}}}(\overline X) = \mathrm{Ric}^{\mathbf h}(X) + (1-r^2)^{-2}\sum_{i=1}^2|A^*_{e_i}U^*|_{\ga}^2 - (1-r^{2})^{-1}|A^*_Xe_3|_{\ga}^2 -3r^{2}(1-r^2)^{-1}\sum_{i=1}^2|A_Xe_i|_{\ga}^2. \qedhere
\end{multline*}
\end{proof}

\section{A semi-Riemannian metric on the standard sphere vs. on the Gromoll and Meyer exotic sphere}

In complete analogy to the previous section, in this section, we deal once more with the bundle diagram \eqref{eq:CDGM} to provide Lorentzian metrics on the classical $\mathrm{S}^7$-sphere as well as in the Gromoll--Meyer sphere $\Sigma_{GM}^7$. As we shall see, the metrics to be provided enjoy the property that time collapse on the existent fixed points in $\mathrm{S}^7$ and $\Sigma_{GM}^7$. Suppose it is the case that compact $7$-dimensional manifolds recover some physical significance (we need to find out if this is the case; however, see \cite{cmp/1103943444} for an account). However, the time-collapsing occurs for the same points up to some identification; see Proposition \ref{prop:isometric}. In this manner, how to perceive a difference due to the choice of a smooth structure needs to be clarified, as we explain next.

\subsection{An induction on the construction of the Gromoll--Meyer exotic sphere}

We review the work initiated by C. Durán in \cite{duran2001pointed} and further developed by L. D. Sperança in his Ph.D. thesis \cite{speranca2012Phd} on producing exotic manifolds in an \emph{equivariant manner}. As we shall see, this can be considered an induction of the Gromoll--Meyer exotic construction described in \eqref{eq:CDGM}. The reference \cite{SperancaCavenaghiPublished} already contains a significative description that we choose to recall once more, bringing the parallels and possibly making more explicit how exoticity emerges from such constructions. More importantly, the feasibility of producing Lorentzian metrics is due to this machinery.

Consider a principal bundle $G\hookrightarrow P \to M$ with principal action $\bullet$ whose total space is compact and connected. Assume that $G$ is also compact and connected and that there is another action by $G$ on $P$, which we denote by $\star$, such that $\bullet$ and $\star$ commute. This makes $P$ a $G\times G$-manifold. If we assume that $\star$ is also free, we get the following $\star$-\textit{diagram} of bundles:
\begin{equation}\label{eq:stardia}
\begin{xy}\xymatrix{& G\ar@{..}[d]^{\bullet} & \\ G\ar@{..}[r]^{\star} & P\ar[d]^{\pi}\ar[r]^{\pi'} &M'\\ &M&}\end{xy}
\end{equation}
The manifold $M$ in the \eqref{eq:stardia} diagram denotes the quotient of $P$ for the action $\bullet$, while $M'$ denotes the quotient of $P$ for action $\star$.

\begin{definition}
The \eqref{eq:stardia} diagram is called $\star$-\textit{(star) diagram}. In short, we will denote it by $M'\leftarrow P \rightarrow M$.
\end{definition}

Following \cite{SperancaCavenaghiPublished}, since $\bullet$ and $\star$ commute, $\bullet$ induces an action on quotient $M'$, and $\star$ induces an action on $M$. We denote orbit spaces according to such induced actions by $M'/\bullet$ and $M/\star$, or, when the context is clear, simply by $M'/G$ and $M/G.$

Let $M$ be a $G$-manifold (on the left) and $\{U_i\}$ be a collection of open sets of $M$ that is $G$-invariant. Reducing $U_i$ if necessary, we can assume that $GU_i = U_i.$
     Given two open $G$-invariant set $U_i,~U_j,$ denoted by $U_{ij} := U_i\cap U_j.$
    
     A collection $\phi_{ij} : U_{ij} \to G$ is said to be a \textit{collection }$\star$ if it satisfies the \textit{cocycle condition},
     \begin{equation}\label{eq:cocyclocondition}
         \phi_{ij}(x)\phi_{jk}(x)\phi_{ki}(x) = \phi_{ii}(x),~\forall x \in U_{ij},
     \end{equation}
     and also,
     \begin{equation}\label{eq:covariance}
         \phi_{ij}(gx) = g\phi_{ij}(x)g^{-1},~\forall i,j,~\forall g\in G,~\forall x\in U_{ij} .
     \end{equation}
    
     The condition \eqref{eq:covariance} guarantees that the \textit{adjoint} map
     \[\widehat{\phi}_{ij} : U_{ij}\to U_{ij} \]
     \begin{equation}
         \widehat{\phi}_{ij} : x\mapsto \phi_{ij}(x)x.
     \end{equation}
     is \textit{equivariant}. In this way, $\widehat{\phi}_{ij}$ defines an equivalent diffeomorphism over $U_{ij}$. 
    
     Let $\{\phi_{ij} : U_{ij}\to G\}$ be a collection $\star.$ We define
     \[\bigcup_{\widehat \phi_{ij}}U_i\]
     as the quotient space associated with the following equivalence relation:
     \[x\in U_{ij}\sim \widehat{\phi}_{ij}x\in U_{ij}.\]
     I.e.,
     given $x,y \in M,$ we say that $x$ is equivalent to $y$ if $\exists i,j$ for which
     $x,y\in U_{ij}$ and yet, $\phi_{ij}(x)x = y.$
    
     The foundation of star diagram construction lies in the next theorem. We will assume the existence of a collection $\star$ on the $G$-manifold $M$ based on a principal bundle $\pi:P \to M.$ In this case, we say that $\pi$ is \textit{linked} to the collection $\star$ defined via $M$.
    
     \begin{theorem}\label{thm:howtobuild}
     Let $\pi : P \to M$ be the principal bundle associated with a collection $\star$ given by $\{\phi_{ij} : U_{ij} \to G\}.$ Then $P $ admits a new action, denoted by $\star$, such that
     \begin{enumerate}
         \item The action $\star$ is free in $P$,
         \item The quotient $P/\star$ is a $G$-manifold equivalently diffeomorphic to
         \[M' := \bigcup_{\widehat{\phi}_{ij}}U_i,\]
         \item The star diagram \eqref{eq:stardia} obtained from this construction is such that if $(G\times G)_p$ denotes the isotropy group on $p$ with respect to the action by juxtaposition
         \[(r,s)p := rps^{-1},\]
         where $G\times \{e\}$ represents the action $\star$, $\{e\}\times G$ represents the principal action, and ``$e$'' denotes the neutral element of $G$ , then there exists $g\in G$ such that
         \begin{equation}
             (G\times G)_p = \{(h,ghg^{-1}) : g\in G_{\pi(p)}\}.
         \end{equation}
     \end{enumerate}
     \end{theorem}

     Let $M'\leftarrow P \rightarrow M$ be a $\star$-bundle. Denote $\bullet$ the principal action in $P.$ Remembering that since $\bullet$ and $\star$ commute, $\star$ induces an action by $G$ in the quotient $M$ and $\bullet $ induces an action by $G$ on $M'$. Thus, we will see $M$ and $M'$ as $G$-manifolds associated with such actions.

An efficient method to compare the geometries of $M$ and $M'$ consists in constructing a metric $G\times G$-invariant on $P$ so that the projections $\pi$ and $\pi'$ induce Riemannian submersion metrics in $M$ and $M'$, respectively. In particular, if we start with a metric $G\times G$-invariant on $P$, if $g_M$ denotes the Riemannian submersion metric induced by $\pi$ on $M$, we prove the following theorem:

\begin{theorem}[Corollary 5.2 in \cite{SperancaCavenaghiPublished}]\label{thm:password}
Let $M'\leftarrow P \rightarrow M$ be a $\star$-bundle. Then, a metric $\bullet$-invariant exists on $M'$ such that $M/G$ and $M'/G$ are isometric as metric spaces.
\end{theorem}

Therefore, let
\[\cal H^{''} := \{X\in \Gamma(TP) : X\perp T(G\times G)\}\]
that is, $\cal H^{''}$ denotes the horizontal space according to the action $(r,s)p \mapsto rps^{-1}.$ If $M$ and $M'$ are provided with Riemannian submersion metrics, then the constraints of $d\pi$ and $d\pi'$ to $\cal H^{''}$ induce isometries between $\cal H^{''}$ and $\cal H\subset TM$, $\cal H'\subset TM',$ respectively.

Remember that the action of $G$ on $P$ induces actions on $TP$ (by the differential of the representation). We abuse notation so that the action induced on $TP$ is written as $gX,~g\in G,~X\in \Gamma(TP).$

  \begin{definition}
A $1$-connection form $\omega_0 : \Gamma(TP) \to \lie g$ in a principal bundle $\pi : G\hookrightarrow P \to M$ consists of a smooth mapping such that
\begin{enumerate}
     \item $\omega_0(U^*) = U,~\forall U\in \lie g,$ where $U^*$ denotes the field of action associated with $U$,
     \item $(\omega_0)_{pg}(Xg) = \mathrm{Ad}_{g^{-1}}(\omega_0)_p(X).$
\end{enumerate}
\end{definition}

Given a $1$-form of connection $\omega : \Gamma(TP) \to \lie g$ and an $G$-invariant metric $g_M$ on $M$, we induce in $P$ a \textit{Kaluza-Klein metric} defined by
\begin{equation}\label{eq:kaluzaklein}
     \langle X,Y\rangle := g_M(d\pi X,d\pi Y) + Q\left(\omega(X),\omega(Y)\right),
\end{equation}
where $Q$ is a bi-invariant metric on $G$. If $g_M$ is $G$-invariant and $\omega$ satisfies
$\omega_{gp}(gX) = \omega_p(X),~\forall g\in G,X\in \Gamma(TP),$ so $\langle\cdot,\cdot\rangle$ is $G\times G$-invariant.

\begin{proposition}\label{prop:oneform}
There is a $1$-connection form $\omega : \Gamma(TP) \to \lie g$ such that for all $X\in \Gamma(TP)$ and all $r\in G $
\[\omega_{rp}(rX) = \omega_p(X).\]
Furthermore, $\langle\cdot,\cdot\rangle$ defined by \eqref{eq:kaluzaklein} is $G\times G$-invariant.
\end{proposition}
\begin{proof}
After taking the \textit{average} of $\omega_0,$
\[\omega_p(X) := \int_G(\omega_0)_{pr}(rX)dr\]
we immediately conclude the result since $\bullet$ and $\star$ commute.
\end{proof}

From now on, we fix a Kaluza-Klein metric $G\times G$-invariant in $P$. As the action of $G\times G$ on $P$ is by isometries, there is only one Riemannian submersion metric on $M'$. Furthermore, this is $G$-invariant. Observe that if $x = \pi(p)$ and $x' = \pi'(p),$ then $\pi^{-1}(Gx) = (G\times G)p = {\pi'}^{-1}(x)$, Thus, a bijection is defined
\[i : M/G \to M'/G.\]
In particular, $i$ is an isometry. Indeed, if $\gamma : [0,1] \to P$ is a horizontal geodesic $\cal H''$, then $\pi \circ \gamma$ and $\pi'\circ \gamma$ are horizontal geodesics to the orbits of $G$ in $M$ and $M'$, respectively. Therefore, given $p_1,p_2\in P$, if $\gamma$ denotes the smallest geodesic segment connecting the orbits $(G\times G)p_1$ and $(G\times G)p_2,$ then $\pi\circ \gamma$ and $\pi'\circ\gamma$ have the same length and are the minimizing geodesics between $G\pi(p_1)$ and $G\pi(p_2)$; $G\pi'(p_1)$ and $G\pi'(p_2);$ respectively. This verifies Theorem \ref{thm:password}.~\qedsymbol{}

\subsection{Lorentzian metric via $\star$-diagrams}

We now use the appropriate Riemannian metrics in $\star$-diagrams to easily promote Lorentzian metrics out of these. Later, we give explicit constructions to Lorentzian metrics on classical and exotic $7$-dimensional spheres.

Take $\ga'$ to be the $\bullet$-invariant metric on $M'$ obtained from Theorem \ref{thm:password}, and let $\ga'_{-r^2}$ its $-r^2$-Cheeger deformation for sufficiently large $r > 0$. Since we do not have any information on whether the orbits of the $\bullet$-action on $M'$ are totally geodesic, the intrinsic geometry of the fibers is decoded by a metric tensor $P'$. So any $r>0$ such that $1-r^2P' < 0$ makes $\ga_{-r^2}$ a semi-Riemannian metric on the principal stratum ${M'}^{reg}$ of $M'$: Recall that outside the principal stratum the orbit tensor may lose rank, so we cannot control the signature of the metric there. 

Observe, in addition, that there is a straightforward way of promoting a semi-Riemannian metric on $M$ for which the orbit spaces $M/G$ and $M'/G$ are still isometric. We claim that it suffices to consider the (same parameter) $-r^2$-Cheeger deformation of $\ga$, a $\star$-invariant metric from which $\ga'$ descends in Theorem \ref{thm:password}. This shall be guaranteed by the following proposition, which we only state and appears in \cite{SperancaCavenaghiPublished}. It is also essential that it shows how $\star$-diagrams encode exoticity through the bundle transitions since, locally, each \emph{local orbit type} is the same for both $M$ and $M'$.

More precisely, for each $x,y \in M$, Theorem 3.57 in \cite{alexandrino2015lie} shows the existence of well-defined $G$-invariant tubular neighborhoods centered on $Gx$ and $Gy$ so that: Denoting by $M^{\sim}_x$ the connected components of the partition on $M$ by orbits with the same type (that is, with conjugate isotropy subgroups) as $Gx,$ we say that any $y,z \in M^{\sim}_x$ has the same local orbit type if there is a $G$-equivariant map between some $G$-invariant tubular neighborhood $\mathrm{Tub}(Gx)$ and $\mathrm{Tub}(Gy)$.

\begin{proposition}[Proposition 5.3 in \cite{SperancaCavenaghiPublished}]\label{prop:isometric}
	Given $x\in M$, there is $x'\in \pi'(\pi^{-1}(x))$ and an isomorphism $\Phi\co \mathrm{Tub}(Gx)\to \mathrm{Tub}(Gx')$ satisfying: given $\cal O\subseteq \mathrm{Tub}(Gx)$, if  $\exp|_\cal O\co \cal O\to M$ is a diffeomorphism onto its image, then $\exp|_{\Phi(\cal O)}\co \Phi(\cal O)\to M'$ is a diffeomorphism onto its image. 
\end{proposition}

Since Cheeger deformations preserve horizontal distributions, we have:
\begin{theorem}
    Let $M'\leftarrow P \rightarrow M$ be a star diagram with $\mathrm{S}^1$ as its structure group. Then there is a Lorentzian metric $\ga_{-r^2}$ in the principal stratum of $M^{reg}$ of the $\star$-action $M$, and a $\ga'_{-r^2}$ Lorentzian metric in the principal stratum ${M'}^{reg}$ of the $\bullet$-action on $M'$ such that the orbit spaces $M/G$ and $M'/G$ are isometric.
\end{theorem}

We finally proceed with explicit constructions.

\subsection{Explicit (almost) Lorentzian metrics on homotopy spheres}

\subsubsection{The $(\mathrm{S}^1)-\star$-action on $\mathrm{S}^7$}
 
	Consider the $\star$ action on $\mathrm{S}^7$. It is defined by the $q$-conjugation on each factor of $(a,b) \in \bb H\oplus \bb H$ satisfying $|a|^2 + |b|^2 = 1$. It is straightforward to check that the vertical space at a point $(a,b) \in \mathrm{S}^7 \subset \mathbb{H}\oplus\mathbb{H}$ is given by
	\begin{equation}
	    \cal V_{(a,b)} = \left\{\begin{pmatrix}[V,a]\\ [V,b] \end{pmatrix},~V \in \lie {\mathbf{i}\bb R}\right\},
	\end{equation}
	\begin{equation}
	    [V, a]:= V a - a V.
	\end{equation}
	
We take $\begin{pmatrix}Xa\\ Yb \end{pmatrix}$ to be orthogonal to the $\star$-orbit through $(a,b).$ Throughout this section  we assume that $\langle\cdot,\cdot\rangle_{\mathrm{S}^7}$ is the restriction of the metric of $\bb H\oplus \bb H$ to $\mathrm{S}^7$. Therefore, $(Xa,Yb)^{T}\in T_{(a,b)}\mathrm{S}^7$ if, and only if, $\mathrm{Re}(|a|^2X + |b|^2Y) = 0$. We prove:
\begin{lemma}\label{lem:horbu}
The vector $\begin{pmatrix}Xa\\ Yb \end{pmatrix}\in T_{(a,b)}\mathrm{\mathrm{S}^7}$ satisfy the equation
\begin{equation}
    \left(|a|^2 - C_{\overline a}\right)X + \left(|b|^2 - C_{\overline b}\right)Y = 0,
\end{equation}
where $C_{\overline a}X := \overline aX a$ and $C_{\overline b}$ is defined similarly. 
\end{lemma}
\begin{proof}
Recall that the Riemannian metric at $(a,b)$ is the sum of quaternionic inner products on $\mathbb{H}\oplus \mathbb{H}$. Therefore, one has:
\begin{align*}
   0 =  \left\langle \begin{pmatrix}Xa\\ Yb \end{pmatrix}, \begin{pmatrix} [V,a]\\ [V,b]\end{pmatrix}\right\rangle &= \mathrm{Re}\left(C_{\overline a}(\overline VX) - \overline VC_{\overline a}(X) + C_{\overline b}(\overline VY) - \overline VC_{\overline b}(Y)\right)\\
   &= \mathrm{Re}\left(|a|^2\overline V X - C_{a}(\overline V)X + |b|^2 \overline VY - C_{b}(\overline V)b\right)\\
   &= \mathrm{Re} \left(\left(\left(|a|^2 - C_{a}\right)\overline V\right)X + \left(\left(|b|^2 - C_{b}\right)\overline V\right)Y\right)\\
   &= \left \langle \left(|a|^2 - C_{\overline a}\right)X + \left(|b|^2 - C_{\overline b}\right)Y,V\right\rangle.\label{eq:deomega}
\end{align*}

One concludes the proof since the last equality must hold for every $V$. \qedhere
\end{proof}

We now consider the $\mathrm{S}^1$-twisted action on $\mathrm{S}^7\times \mathrm{S}^1$ defined as
\begin{equation}\label{eq:twists7}
q\ast ((a,b)^{T},q') := ((qa\overline q,qb\overline q)^{T},qq'),
\end{equation}
where $(a,b)^{T} \equiv \begin{pmatrix}
a\\
b
\end{pmatrix}.
$

In complete analogy to Lemma \ref{lem:horbu} one has that the horizontal space at $((a,b),q')$ with respect to the action \eqref{eq:twists7}, computed in the metric

\begin{equation*}
\widetilde{\ga}_{-r^{-2}} := \langle\cdot,\cdot\rangle_{\mathrm{\mathrm{S}^7}} -r^{-2}d\theta^2,
\end{equation*}
 is given by
 \begin{equation*}
\cal H_{((a,b),q')} \cong \mathrm{span}_{\bb R}\left\{((Xa,Yb)^{T},iq't) \in T_{(a,b)}\mathrm{S}^7\oplus T_{q'}\mathrm{S}^1 : (|a|^2-C_{\overline a})X + (|b|^2-C_{\overline b})Y = tr^{-2}\right\}.
 \end{equation*}

 Therefore, once more identifying $\partial_{\theta} \leftrightarrow \mathbf{i}$ one gets the following description of the tangent space of $\mathrm{S}^7$ at $(a,b)$
 \begin{equation*}
T_{(a,b)}\mathrm{S}^7 \cong \mathrm{span}_{\bb R}\left\{t\partial_{\theta}, (Xa,Yb)^{T} : (Xa,Yb) \in T_{(a,b)}\mathrm{S}^7, (|a|^2-C_{\overline a})X + (|b|^2-C_{\overline b})Y = tr^{-2} \right\},
 \end{equation*}
 with the following semi-Riemannian metric:
 \begin{equation*}
\ga_{-r^{2}}(t\partial_{\theta} + (Xa,Yb)^{T},t'\partial_{\theta} + (X'a,Y'b)^{T}) = \langle X,X'\rangle_{\bb H} + \langle Y,Y'\rangle_{\bb H} - r^{-2}tt'.
 \end{equation*}

\subsubsection{The $(\mathrm{S}^1)-\bullet$-action on $\Sigma^7_{GM}$}
	
	Consider the $\bullet$ action on $\Sigma_{GM}^7$. It is given by the conjugation by $q\in \mathrm{S}^1$ only in the first factor $(2\bar cd,|c|^2-|d|^2)$. Hence, the vertical space at a point $(2\overline cd,|c|^2-|d|^2) \in \Sigma^7 \subset \mathbb{H}\oplus\mathbb{H}$ is given by
	\begin{equation}
	    \cal V_{(2\overline cd,|c|^2-|d|^2)} = \mathrm{span}_{\bb R}\left\{\begin{pmatrix}2\mathbf{i}\overline cd - 2\overline cd\mathbf{i}\\ 0\end{pmatrix}\right\}.
	\end{equation}
	
Take $\begin{pmatrix} \h X\overline cd\\ Y \end{pmatrix}\in T_{(2\overline cd,|c|^2-|d|^2)}\Sigma_{GM}^7$ to be $\widetilde{\ga}_{-r^{2}}$-orthogonal to the $\star$-orbit through $(2\overline cd,|c|^2-|d|^2)$, where
\[\widetilde{\ga}_{-r^{2}} = \langle\cdot,\cdot\rangle_{\bb H\oplus \bb H} -r^{2}d\theta^2.\]
Observe that $(\h X\overline cd,Y) \in T_{(2\overline cd,|c|^2-|d|^2)}$ if, and only if,
\[|\bar c d|^2\mathrm{Re}(X) = -\mathrm{Re}(Y) (|c|^2-|d|^2).\] Moreover, copying the proof of Lemma \ref{lem:horbu} we can check that
\begin{lemma}
The vector $\begin{pmatrix}\h X\overline cd\\ Y \end{pmatrix}$ satisfy the equation
\[|\bar cd|^2 X- C_{\bar cd}X = 0,\]
where $C_{\bar cd}$ is the conjugation by $\bar c d$.
\end{lemma}

We now consider the $\mathrm{S}^1$-twisted action on $\Sigma^7_{GM}\times \mathrm{S}^1$ defined as
\begin{equation}\label{eq:twistsigma7}
q\ast ((2\overline cd,|c|^2-|d|^2)^{T},q') := ((q\overline cd\overline q,|c|^2-|d|^2)^{T},qq'),
\end{equation}
where $(2\overline cd,|c|^2-|d|^2)^{T} \equiv \begin{pmatrix}
2\overline cd\\
|c|^2-|d|^2
\end{pmatrix}.
$

Therefore, it is straightforward to check that the horizontal for $\ast$ is given by
\begin{equation*}
\cal H_{((\bar cd,|c|^2-|d|^2),q')}\cong \left\{\left(\left(-\h X\overline cdi,Y\right),itq'\right) : |\bar c d|^2X + C_{\bar cd}X = r^{-2}t, ~|\bar c d|^2\mathrm{Re}(X) = -\mathrm{Re}(Y)(|c|^2-|d|^2) \right\}.
\end{equation*}
Hence, any horizontal vector has the norm
\begin{equation*}
\left|\left(\left(-\h X\overline cd\mathbf{i},Y\right),\mathbf{i}tq'\right)\right|^2_{\widetilde{\ga}_{-r^{2}}} = \frac{1}{4}|X|^2|c|^2|d|^2 + |Y|^2 -t^2r^{-2}
\end{equation*}
and so the Lorentzian metric induced on $\Sigma_{GM}^7$ can be defined as
\begin{equation}
    \ga_{-r^2}(-\h X\overline cd\mathbf{i} + Y,-\h X'\overline cd\mathbf{i} +Y') := \frac{1}{4}|c|^2|d|^2\langle X,Y\rangle_{\bb H} + \langle X',Y'\rangle_{\bb H} - t^2r^{-2}.
\end{equation}

Observe that the $\mathrm{S}^1$-action constructed in $\mathrm{S}^7$ has some fixed points, the same holding for the $\mathrm{S}^1$-action built on $\Sigma_{GM}^7$. In this manner, just as in the case of the $\mathrm{S}^1$-action constructed in $\mathrm{SU}(2)$, these metrics present time singularities, which we emphasize once more, could be interpreted as restricted world perspectives of observers living in $\mathrm{S}^7\times \bb R$ or $\Sigma_{GM}^7\times \bb R$, the corresponding universal coverings of $\mathrm{S}^7\times \mathrm{S}^1$ and $\Sigma_{GM}^7\times \mathrm{S}^1$. In the next section, we provide a better study of such singularities besides furnishing a classification.

\section{Fixed points singularities and Ricci curvature}
\label{sec:fixing}
\begin{lemma}\label{lem:fixedpointssu2}
Let $\mathrm{S}^7$ with the $\mathrm{S}^1$-action defined by equation \eqref{eq:GMstaraction}. Then, the only fixed points $(a,b)\in \mathrm{S}^7$ by the $\mathrm{S}^1$-action are \[(a,b)\in \mathrm{S}^7\cap \left(\bb R\times \bb R\right). \]
Moreover, for each $(a,b) \in \mathrm{S}^7$ such that $a\in \mathrm{S}^1, b = 0$ or $b\in \mathrm{S}^1, a = 0$, we have that $\mathrm{S}^1_{(a,b)} = \{e,(a,b)\}\cong \mathbb{Z}_2$ and out of these points the isotropy subgroups are trivial.
\end{lemma}
\begin{proof}
Since we are considering $(a,b)\in \mathrm{S}^7$ where both $a, b$ are quaternions, let $q\in \mathrm{S}^1$ be such that $q\neq 1$ but $qa\bar q = a, qb\bar q = b$. It follows that $[a,q] = 0 = [b,q]$, where $[\cdot,\cdot]$ denotes the quaternionic bracket. Assuming that $\mathrm{Im}a, \mathrm{Im}b \neq 0$ write $a = |a|\mathrm{e}^{\mathrm{Im}a\theta_a/|\mathrm{Im}a|}$, $b = |b|\mathrm{e}^{\mathrm{Im}b\theta_b/|\mathrm{Im}b|}$ and $q = \mathrm{e}^{\mathbf{i}\theta_q}$. Here, $\theta_{a}, \theta_{b}, \theta_{q} \in [0,2\pi[$ are such that
\[\cos(\theta_a) = \mathrm{Re}a/|a|, ~\sin(\theta_a) = |\mathrm{Im}a|/|a|,\]
making the corresponding analogous definitions for both $b, q$. 

Hence,
\begin{equation*}
[a,q] = 0 \Leftrightarrow \mathrm{e}^{2\left(\mathrm{Im a}\theta_a/|\mathrm{Im} a| - i\theta_q\right)} = 1.
\end{equation*}

From the definition of exponential to quaternions, one gets that $\mathrm{Im~a}\theta_a = |\mathrm{Im}~a|\mathbf{i}\theta_q$.
Hence, $a = a_0 + |\mathrm{Im}a|\theta_q/\theta_a\mathbf{i} + 0\mathbf{j} + 0\mathbf{k}$. The same argumentation implies that $b = b_0 + |\mathrm{Im b}|\theta_q/\theta_b\mathbf{i} + 0\mathbf{j} + 0\mathbf{k}$. Hence, $\theta_q = \theta_a = \theta_b$ and so $q = a_0/|a| + \mathbf{i}a_1/|a| = b_0/|b| + \mathbf{i}b_1/|b|$.
However, using that $|a|^2 + |b|^2 = 1$ we have that $|q|^2 = 2|a|^2+2|b|^2 = 2$, a contradiction. Therefore, either $\mathrm{Im}a = 0$ or $\mathrm{Im}b = 0$. In the first case, $a\in\mathbb{R}$ and $b = b_0 + |\mathrm{Im}b|\mathbf{i}$ since $\theta_q = \theta_b$. Moreover, $q = b_0/|b|+b_1\mathbf{i}/|b|$ and we once more derive the same contradiction unless $\mathrm{Re}a = 0$. Hence, either $\mathrm{Im}b = 0$ or $a = 0$. That considered, we have that
$(a,b) \in \mathrm{S}^7$ such that $a,b \in \bb R$ are fixed points by the action defined in equation\eqref{eq:GMstaraction}.

Assume now that $b = 0$ but $\mathrm{Im}a \neq 0$. Then $\mathrm{Im a}\theta_a = |\mathrm{Im}a|\mathbf{i}\theta_q$ and $\theta_a = \theta_q$, that is, in this case $a = q$. So $\mathrm{S}^1$-itself is fixed by this action. The same analysis holds for $a = 0$ but $b\neq 0$, finishing the proof.\qedhere
\end{proof}

 At these points, we now study the behavior of what is called \emph{fake horizontal vectors}, first introduced in \cite{cavenaghiesilva}. Namely, since the vertical space collapses at fixed points, say that $x\in M$ is one such point. Then, fixed a horizontal geodesic $c$ emanating from $x$, we proceed to characterize the horizontal vectors at $x$, which behave like time-like vectors along $c$ for $s \neq 0$.

 \begin{lemma}[Lemma 3.2 in \cite{cavenaghiesilva}]\label{lem:S_X}
Let $(M,\ga)$ be a path-connected Riemannian manifold with an effective isometric action by a compact connected Lie group $G$. Let $x$ be a $G$-fixed point and $X\in \cal H_x$. Assume that the horizontal geodesic $\gamma : [0,\epsilon] \to M$ defined as $\gamma(s)=\exp_x(sX)$ intersects the principal stratum for any $s > 0.$ Define \begin{align*}
    \tilde S_X:\lie g&\to T_xM\\ U&\mapsto \nabla_XU^*_x.
  \end{align*} 
  Then,
    \begin{enumerate}[$1.$]
      \item the image $\tilde S_X({\lie g_x})$ is contained in $\cal H_x$. Moreover, for $\epsilon>0$ sufficiently small, the following defines a smooth bundle on $\gamma([0,\epsilon))$
      \[\tilde{\cal H}_s=\begin{cases}\cal H_{\gamma(s)} ~~\text{if}~ s > 0\\ (\tilde S_X(\lie g_x))^\bot~ ~\text{if}~ s=0.
      \end{cases}
      \]
      \item the kernel of the restriction $\tilde S_X|_{\lie g_X}$ coincides with $\lie g_X$, the Lie algebra of $G_X=\{r\in G_x~|~\rho(r)X=X \}$, where $\rho(r) : T_xM\rightarrow T_xM$ is the isotropy representation at $r\in G_x$.
    \end{enumerate} 
  \end{lemma}
   
  In our case, in a fixed point $x$ we have that $\lie g_x = \lie g$ but also that $\dim_{\bb R}\lie g_x = \dim_{\bb R}\mathbf{i}\bb R =1$. Therefore, for each $X\in \cal H_x$ either $\lie g_X = 0$ or $\lie g_X \cong \lie g \cong \mathbf i\bb R$. If $\lie g_X = \mathbf i \bb R$, then $\widetilde S_X \equiv 0$ and hence, the parallel transport of each horizontal vector along the geodesic generated by $X$ leads to a horizontal vector at $x$. Hence, $X$ is precisely a \emph{fake-horizontal vector}. The geodesic $s\mapsto \gamma(s)$ is time-like for $s > 0$. We conclude:
  \begin{theorem}\label{thm:vaiajudar}
      Let $(M,\ga_{-r^2})$ be a path-connected semi-Riemannian manifold with an effective isometric action by $\mathrm{S}^1$ and $x$ be a $G$-fixed point. Assuming that $\ga_{-r^2}$ is Lorentzian in the principal stratum, given $X\in \cal H_x \cong T_xM$, the geodesic generated by $s\mapsto \gamma(s)$ is time-like for any $s>0$ if, and only if, $\lie g_X = \lie{s}^1\cong \mathbf i\bb R$.
  \end{theorem}

  We can fix the metric $\ga_{-r^2}$ with the former theorem to produce a global Lorentzian metric on any almost Lorentzian manifold.
  \begin{theorem}
      Let $(M,\ga)$ be an almost Lorentzian manifold with the Riemannian metric $\ga_{-r^2}$. The there exists a Lorentzian metric $\widetilde \ga_{-r^2}$ on $M$ which coincides with $\ga_{-r^2}$ in the principal stratum in $M$.
  \end{theorem}
  \begin{proof}
      According to Theorem \ref{thm:vaiajudar}, let $x\in M$ be a fixed point for the $\mathrm{S}^1$ almost semi-free action. Take $X\in T_xM$ such that $\lie g_X = \lie{s}^1$. We define
      \[\widetilde \ga_{-r^2} = \ga|_{\left(\mathrm{span}_{\bb R}\left\{X\right\}\right)^{\perp_{\ga}}} +(1-r^2)^{-1}\ga|_{\mathrm{span}_{\bb R}\{X\}}.\]
      Moreover, we let $\widetilde \ga_{-r^2} := \ga_{-r^2}$ be the Lorentzian metric in the principal stratum $M^{reg}$. The well-definition of this metric is straightforward. \qedhere
  \end{proof}

To exemplify the desired fake horizontal vector at a fixed point given by Lemma \ref{lem:fixedpointssu2}, let $(a,b)\in \bb R\oplus \bb R$ such that $|a|^2 + |b|^2 = 1$. Then, the $\mathrm{S}^1$ action by conjugation in each factor fixes $(a,b)$. Let $s\mapsto c(t) = (a(s),b(s))$ a smooth curve in $\mathrm{SU}(2)$ such that $(a(0),b(0)) = (a,b)$ and assume that $(a'(0),b'(0)) = (Xa,bY)$ where $\mathrm{Re}(X|a|^2+Y|b|^2) = 0$. Then,
  \begin{align*}
      q\ast (a(s),b(s)) &:= (qa(s)\bar q,qb(s)\bar q)\\
      \rho(q)(Xa,bY) &= (qaX\bar q,qbY\bar q)\\
      \rho(q)(Xa,bY) &= (Xa,bY) \Leftrightarrow qaX=Xaq,~qbY = byq\\
      &\Leftrightarrow aX, bY \in \mathbb{C} \cong \mathrm{span}_{\bb R}\left\{1,\mathbf{i}\right\}.
  \end{align*}
  Hence, the geodesic generated by $(aX,bY) \in T_{(a,b)}\mathrm{SU}(2)\cap \bb C\oplus \bb C$ for $(a,b)\in \mathrm{SU}(2)\cap (\bb R\times \bb R)$ is time-like.

\subsection{$\mathrm{S}^1$-Fat bundles and a dual leaf type result}

	We recall that for a general (non-necessarily metric) submersion, the \emph{fatness} condition is an intrinsic property of its horizontal distribution $\cal H$. Namely, a submersion $\pi: M \rightarrow B$ is said to be \textit{fat} if, for every non-zero $X\in\cal H$, $\cal V=[\widetilde X,\cal H]^{\mathbf v}$ for some horizontal extension of $X$, where we denote by $\cal V$ the vertical distribution associated to $\pi$. We observe that often in the literature, the nomenclature for fatness is widely employed, assuming the preexistence of a Riemannian metric with totally geodesic leaves.

 Treating the fatness condition for $\mathrm{S}^1$-principal bundles is quite different from other structure groups simply due to the fact $\lie{s}^1$ is the (1$\bb R$-dimensional) Abelian Lie algebra generated by $\mathbf{i} \in \bb{C}$. In this manner, any connection form $\omega$ on the bundle and its corresponding curvature form $\Omega$ are nothing but ordinary $1$, $2$-forms, respectively, where $d\omega =\Omega$.

It can be inferred from Chern-Weil theory (\cite[Corollary 2.11, p. 10]{Ziller_fatnessrevisited}) that a principal $\mathrm{S}^1$-bundle is fat if, and only if the curvature form $\Omega$ is a pullback of some symplectic form on the base $B$: Fat $\mathrm{S}^1$-principal bundles are (bi)univocally associated with symplectic manifolds.

We recall the following examples of $\mathrm{S}^1$-fat Riemannian submersions (see \cite{Ziller_fatnessrevisited})
\begin{enumerate}[(a)]
    \item the Hopf bundles $\mathrm{S}^1\hookrightarrow \mathrm{S}^{2n+1}\rightarrow \bb{CP}^n$
    \item any $\mathrm{S}^1$-principal Riemannian bundle $(M,\ga) \rightarrow (B,\ga_B)$ with fibers of length $2\pi$ such that $\ga_B$ is either
    \begin{enumerate}[(i)]
        \item almost K\"ahler and so $\ga$ is $K$-contact
        \item K\"ahler and so $\ga$ is Sasakian
        \item K\"ahler Einstein and so $\ga$ is Sasakian Einstein
    \end{enumerate}
\end{enumerate}
A metric on a smooth manifold $M$ is named $K$-contact if it admits a unit Killing field $V$ with $\mathrm{sec}_{\ga}(X, V)  = 1$ for every non-degenerated vertizontal plane $X\wedge V$. We say further that the metric is regular if the integral curves of $V$ define a free-circle action, and it is quasi-regular if they define an almost free-circle action. According to \cite{Wadsley}, the latter is equivalent to requiring that all integral curves of $V$ are closed. If it is regular, $M\rightarrow B$ is a fat Riemannian submersion with totally geodesic fibers. 

 A metric is called Sasakian if it admits a unit Killing vector field $X$ such that $R_{\ga}(V, X)Y = \ga(V, Y)X-\ga(X, Y)V$ for every $X, Y\perp V$. As before, the Sasakian structure is called regular if the integral curves of $V$ define a free circle action and quasi-regular if they define an almost free circle action.

\subsubsection{Holonomy and Dual Holonomy fields}

Let $(M,\ga)$ be a Riemannian manifold with a Riemannian foliation $\cal F$. Considering the natural decomposition $TM = \cal V\oplus \cal H$ where $\cal V$ stands for the sub-bundle of $TM$ containing the vectors tangent to the leaves of $\cal F$, if $c : I \rightarrow M$ is a geodesic such that $\dot c(s)\in \cal H_{c(s)}$ for any $s\in I \subset \mathbb{R}$ a vector field $\nu$ such that $\nu(s) \in \cal V_{c(s)}$ for every $s\in I$ is named a dual-holonomy field provided if
\begin{equation}\label{eq:dualnaintro}
\nabla_{\dot c}\nu = S_{\dot c}\nu - A^*_{\dot c}\nu.
\end{equation}
Analogously, basic horizontal fields along a vertical curve $\gamma$ are horizontal solutions of
\begin{equation}\label{eq:codualnaintro}
\nabla_{\dot \gamma}X = -A^*_{X}\dot \gamma - S_{X}(\dot \gamma).
\end{equation}
In the above equations, $S_X$ always stands for the second fundamental form of a leaf. The tensor $A^*_X$ is $\ga$-dual of the O'Neill tensor $A_XY = \h[X,Y]^{\mathbf v}$, $X,Y \in \cal H$. Under the totally geodesic fibers assumption, both dual-holonomy and holonomy fields along horizontal geodesics and basic horizontal fields
along vertical geodesics are the Jacobi fields induced
by, respectively, local horizontal lifting and holonomy transport.

Moreover, $A^*_X\dot \gamma$ is basic along the geodesic $\exp(sV)$, where $V\in \cal V$ and $X\in \cal H$ is basic. If $\ga$ has non-negative sectional curvature, then for each leaf $L_p \in \cal H$,
\[T_pL_p^{\#}\cap \cal V_p = \mathrm{span}_{\mathbb{R}}\{A_XY(p) : X, Y \in \cal H_p\},\]
where $L_p^{\#}$ is a dual leaf through $p$:
\[L^{\#}_p := \{q \in M : \text{there exists a piecewise smooth horizontal curve from $p$ to $q$}\},\]
as described in \cite{wilkilng-dual}. 

Suppose that the foliation $\cal F$ is given by the connected components of the fibers of a Riemannian submersion $\pi: M \rightarrow B$. Recall that the \emph{holonomy group} at some $b \in B$ is defined considering: Fixed $b \in B$, consider all simple closed curves in $B$ based at $b$. We can define a diffeomorphism of the fiber $F = \pi^{-1}(b)$ over $b$ by lifting the curve horizontally to points of $F$. The collection of all of these diffeomorphisms forms an appropriate composition rule. Such a group is trivial in the case of a Riemannian product $F\times B \rightarrow B$. Generally, the holonomy group is not a Lie group, but it certainly is for \emph{Homogeneous bundles}. This is precisely the case for fat Riemannian submersions. 

For many foliations, such as the Riemannian foliation induced by the isometric action of a compact connected Lie group, it is known that the dual-foliation is a \emph{Singular Riemannian Foliation} (see \cite{wilkilng-dual}). A crucial point to the verification of this is the usage of completeness of each dual leaf.

On the other hand, it is known (for instance, via the now classic example of J. Wolf (\cite[Corollary 11.1]{Wolf1964})) that even under very high symmetrical conditions, a Lorentzian manifold needs not to be geodesically complete. We can show, however, that the family of compact Riemannian manifolds with almost semi-free $\mathrm{S}^1$-actions, when regarded with the metric $\widetilde \ga_{-r^2}$ given by Theorem \ref{thm:vaiajudar}, is time-like and space-like complete. In particular, we show that the dual leaf foliation induced by a fixed horizontal distribution has complete leaves for $\mathrm{S}^1$-fat foliations. Indeed, the dual foliation has a single leaf, which coincides with $M$. This was already observed for Riemannian metrics in \cite{cavenaghi2023dual}, so we show it works as well for Lorentzian metrics induced by circle actions.

\begin{theorem}
    Let $\mathrm{S}^1\rightarrow M \rightarrow B$ be a principal $\mathrm{S}^1$-fat bundle with connection form $\omega$. If $\ga$ is the unique (up to scaling) connection metric $\ga$ on $M$, any (semi) Riemannian dual foliation for the Lorentzian metric $\widetilde \ga_{-r^2}$ (for $r>1), $ has only one leaf, which coincides with $M$. Moreover, if the action is almost semi-free, the conclusion holds in the principal stratum $M^{reg}$.
\end{theorem}
\begin{proof}
For the proof, note that if $\ga$ is the unique (up to scaling) connection metric on $\mathrm{S}^1\rightarrow M \rightarrow B$ and $A: \cal H\times \cal H \rightarrow \cal V$ is the O'Neill tensor, denoting by $A^*_X$ the $\ga$-dual to $A_X$ for any fixed $X\in \cal H$ one has that ${A^*_{-r^2}}_XV = (1-r^2)^{-1}A^*_XV$, where ${A^*_{-r^2}}_X$ is the $\ga_{-r^2}$-dual tensor to $A_X$. In particular, since the orbits of $\mathrm{S}, ^1$ are totally geodesic for $\ga$, these are totally geodesic for $\ga_{-r^2}$ as well. So $\mathrm{sec}_{\ga_{-r^2}}(X,V) = -(1-r^2)^{-1}|A^*_XV|_{\ga}^2$ for any $\ga_{-r^2}$-orthonormal vertizontal plane. In this manner, the fatness condition ensures that $\mathrm{sec}_{\ga_{-r^2}}(X,V) < 0$. 

We claim that each dual leaf is open. Indeed, since the dual foliation only depends on the choice of a horizontal distribution, we take $\cal H$ as the horizontal distribution given by the $\ga$-orthogonal complement of the Riemannian foliation induced by the orbits of $\mathrm{S}^1$. Now take a point $x\in L_x^{\#}$, where $L_x^{\#}$ is the dual leaf through $x$. The fatness condition ensures that any given $V\in \cal V_x$ has a constant norm along any horizontal geodesic emanating from $x$. This follows easily from equation (11) in \cite{speranca2017on}: Since $\mathrm{sec}_{\ga_{-r^2}}(X,V) < 0$ then any dual holonomy $\nu$  field along the geodesic $c$ generated by $X$ has constant norm, since
$\h\frac{d^2}{dt^2}|\nu|^2 < \mathrm{sec}_{\ga}(X,V)$
implies that $|\nu|^2 = e^{t\lambda}$ for $\lambda \in (-\sqrt{2},\sqrt{2})$. If $\lambda \neq 0$ then $|\nu|^2 \rightarrow \infty$ for $t\nearrow \infty$ (if $\lambda > 0$) or $-t\nearrow \infty$ (if $\lambda < 0$). Both cases contradict the compactness of $\mathrm{S}^1$, since every compact structure group enjoys the property $|\nu(t)|\leq L|\nu(0)|$ for some $L > 0$, see Definition 1.3 in \cite{speranca2017on}. The claim holds since then the dimension of $T_{c(t)}L_{c(t)}^{\#}\cap \cal V_{c(t)}$ is constant and so $\dim M = \dim L_x^{\#}$. \qedhere
\end{proof}

Recall that if $\mathrm{S}^1$ acts almost freely on a smooth compact connected manifold $M$, then the quotient space $M/G$ has the structure of an orbifold, see \cite[Proposition 1.5.1, p. 17]{caramello2022introduction}. In the following corollary, we use the terminology \emph{orbibundle} to refer to a smooth manifold with an effective isometric, almost free action, so the quotient is seen as an orbifold.

\begin{corollary}
The following manifolds admit Lorentzian metrics, which are time-like and space-like geodesically complete:
\begin{enumerate}[(a)]
    \item the Hopf orbibundle $\mathrm{S}^1\hookrightarrow \mathrm{S}^{3}\rightarrow \bb{CP}^1$, 
    \item any $\mathrm{S}^1$-principal Riemannian orbibundle $(M,\ga) \rightarrow (B,\ga_B)$ with (non-trivial) fibers of length $2\pi$ such that $\ga_B$ is either
    \begin{enumerate}[(i)]
        \item almost K\"ahler and so $\ga$ is $K$-contact
        \item K\"ahler and so $\ga$ is Sasakian
        \item K\"ahler Einstein and so $\ga$ is Sasakian Einstein
    \end{enumerate}
\end{enumerate}
\end{corollary}

\subsection{A general theorem for Lorentzian manifolds with positive Ricci curvature metrics}

We finish this section proving that our manifolds in this paper have positive Ricci curvature—the same holding for the product space that furnishes the considered Lorentzian metrics via Cheeger deformations.

\begin{theorem}\label{thm:einsteingeneral}
Let $(M,\ga)$ be a closed Riemannian manifold of dimension $n\geq 3$ with non-negative sectional curvature. Assume a free isometric $\mathrm{S}^1$-action on $M$ exists. Then for any $r>1$, the Lorentzian metric $\mathrm{g}_{-r^2}$ has positive Ricci curvature, the same holding for the Lorentzian manifold $(M\times \mathrm{S}^1,\ga-r^{-2}Q)$, where $Q$ is the standard bi-invariant metric on $\lie{s}^1$.

If the $\mathrm{S}^1$-action is almost semi-free, there exists $r^2$ sufficiently large such that the metric $\ga_{-r^2}$ has positive Ricci curvature in $M$. In particular, the metric $\widetilde \ga_{-r^2}$ given by Theorem \ref{thm:vaiajudar} has positive Ricci curvature.
\end{theorem}
\begin{proof}
    The proof is a generalization of the proof of Theorem \ref{thm:einsteintoy}. We take $\mathrm{S}^1$ to be totally geodesic for $\ga$, and so the orbit tensor $P: \lie m_x \rightarrow \lie m_x$, for every $x\in M$, can be taken equal to the identity, where we decompose
    $\lie g = \lie m_x\oplus \lie g_x$ orthogonaly with respect to any fixed bi-invariant metric induced from $\mathrm{Iso}(\ga)$. If the $G$-action is free, then $\lie m_x \cong \lie g$ for every $x\in M$. Otherwise, in the presence of fixed points $\lie g_x = \lie g$ for every fixed point $x$ and $P$ collapses to the $0$-operator along any geodesic coming from the principal stratum: see Lemma \ref{lem:S_X}.

    We compute the Ricci curvature of the metric $\ga_{-r^2}$. With this aim, we write any vector $\overline X\in T_xM$ as $\overline X = X + U^*$ for $U\in \mathbf{i}\bb R$ where the superscript $\ast$ reinforces that $U^*$ is an action vector at $x\in M$. Taking $\{e_1,\ldots,e_{n-1}\}$ as a $\ga$-orthonormal basis to $\cal H_x$ and $\{e_n\}\in \mathbf{i}\bb R$ is a unit vector, one gets that $\{e_1,\ldots,e_{n-1}\}\cup \left\{\sqrt{-1}(1-r^2)^{1/2}e_n\right\}$ is a $\ga_{-r^{2}}$-orthogonal basis to $T_xM$ such that $\ga_{-r^2}(\sqrt{-1}(1-r^2)^{1/2}e_n,\sqrt{-1}(1-r^2)^{1/2}e_n) = -1$. We have added the $\sqrt{-1}$ multiplying factor since we are considering $r^2-1>0$. So $(1-r^2)^{1/2}$ is a complex number. Therefore, just as in the proof of Theorem \ref{thm:einsteintoy}, this choice of basis works as a simple way to write the needed basis explicitly.

    Now the needed unreduced sectional curvature can be easily computed for any fixed $\overline X \in T_xM$ as for $i\in \{1,\ldots,n-1\}$
    \begin{align*}
        K_{\ga_{-r^2}}(\overline X,e_i) &= \kappa_{-r^2}((1-r^2)^{-1}\overline X,e_i)\\
        &= K_{\ga}((1-r^2)^{-1}\overline X,e_i) -3r^2(1-r^2)^{-1}|\nabla_{(1-r^2)^{-1}\overline X}e_i|_{\ga}^2\\
        &=K_{\ga}(X,e_i) + (1-r^2)^{-2}K_{\ga}(U^*,e_i) + 2(1-r^2)^{-1}R_{\ga}(e_1,X,U^*,e_i)-3r^2(1-r^2)^{-1}|\nabla_{(1-r^2)^{-1}\overline X}e_i|_{\ga}^2\\
        &= K_{\ga}(X,e_i) + (1-r^2)^{-2}|A^*_{e_i}U^*|_{\ga}^2-3r^2(1-r^2)^{-1}|\nabla_{(1-r^2)^{-1}\overline X}e_i|_{\ga}^2
    \end{align*}
    where we have used the 'vertizontal' curvature under totally geodesic fibers is given by $|A^*_{e_i}U^*|_{\ga}^2$ and that the mixed terms curvature component $R_{\ga}(e_1, X, U^*,e_i)$ vanishes.

    Moreover,
    \begin{align*}
        K_{\ga_{-r^2}}(\overline X,\sqrt{-1}(1-r^2)^{1/2}e_n) &= \kappa_{-r^2}((1-r^2)^{-1}\overline X,\sqrt{-1}(1-r^2)^{-1/2}e_n)\\
        &= K_{\ga}((1-r^2)^{-1}\overline X,\sqrt{-1}(1-r^2)^{-1/2}e_n) - 3r^2(1-r^2)^{-1}|\nabla_{(1-r^2)^{-1}\overline X}\sqrt{-1}(1-r^2)^{1/2}e_n|_{\ga}^2\\
        &= -(1-r^2)^{-1}K_{\ga}(X,e_n) -3r^2(1-r^2)^{-1}|\nabla_{(1-r^2)^{-1}\overline X}\sqrt{-1}(1-r^2)^{1/2}e_n|_{\ga}^2\\
        &= -(1-r^2)^{-1}|A^*_Xe_n|_{\ga}^2 -3r^2(1-r^2)^{-1}|\nabla_{(1-r^2)^{-1}\overline X}\sqrt{-1}(1-r^2)^{1/2}e_n|_{\ga}^2.
    \end{align*}

    The equation on the last line is obtained again due to the vanishing of mixed curvature terms and the fact that the fiber is one-dimensional. This also implies that
    \[\nabla_{(1-r^2)^{-1}\overline X}\sqrt{-1}(1-r^2)^{1/2}e_n = 0.\]
    Just observe that the possible non-vanishing term is $\sqrt{-1}(1-r^2)^{-1/2}\nabla_Xe_n$ which coincides with (a multiple of) the shape operator of the fibers, that is $0$ due to the totally geodesic assumption. Hence,
    \begin{equation}
          K_{\ga_{-r^2}}(\overline X,\sqrt{-1}(1-r^2)^{1/2}e_n) = -(1-r^2)^{-1}|A^*_Xe_n|_{\ga}^2.
    \end{equation}

    The same argument to deal with $\nabla_{(1-r^2)^{-1}\overline X}e_i$ can be employed ensuring that 
    \[\nabla_{X+(1-r^2)^{-1}U^*}e_i = \nabla_Xe_i = A_Xe_i.\]
    Therefore,
 \[
        \mathrm{Ric}_{\ga_{-r^2}}(\overline X) = \mathrm{Ric}^{\mathbf h}(X) + (1-r^2)^{-2}\sum_{i=1}^{n-1}|A^*_{e_i}U^*|^2_{\ga}-3r^2(1-r^2)^{-1}\sum_{i=1}^{n-1}|A_Xe_i|_{\ga}^2\\
        -(1-r^2)^{-1}|A^*_Xe_n|_{\ga}^2.
\]
    Since $\mathrm{Ric}^{\mathbf h}(X) = \mathrm{Ric}_{M/G}-3\sum_{i=1}^{n-1}|A_Xe_i|_{\ga}^2$ we have
     \[
        \mathrm{Ric}_{\ga_{-r^2}}(\overline X) = \mathrm{Ric}_{M/G}(X) + (1-r^2)^{-2}\sum_{i=1}^{n-1}|A^*_{e_i}U^*|^2_{\ga}-3(1-r^2)^{-1}\sum_{i=1}^{n-1}|A_Xe_i|_{\ga}^2\\
        -(1-r^2)^{-1}|A^*_Xe_n|_{\ga}^2 \geq 0. 
\]

    We prove that $\ga -r^{-2}Q$ has positive Ricci curvature. With this aim, we observe that the vertical space associated with the $\star$-acion defined by equation \eqref{eq:star} on $M\times \mathrm{S}^1$ has as vertical space at $(x,q)\in M\times \mathrm{S}^1$ the space
    \begin{equation*}
        \cal V_{(x,q)} = \left\{ (U^*,U) : U \in \mathbf{i}\bb R\right\}.
    \end{equation*}
Moreover, its $\ga-r^{-2}Q$-orthogonal complement $\cal H^{\star}_{(x,q)}$ can be described as
\begin{equation*}
    \cal H^{\star}_{(x,q)} = \left\{(X+r^{-2}W^*,W) : X\in \cal H_x,~W\in \mathbf{i}\bb R \right\}.
\end{equation*}

Observe that if one takes $\{e_1,\ldots,e_{n-1}\}$ to be a $\ga$-orthonormal set spanning $\cal H_x$ and a unit vector $e_n\in \mathbf{i}\bb R$ (for $\ga|_{\lie \mathrm{S}^1} = Q$) then the following constitute a desired basis for computing the Ricci curvature of $\ga -r^{-1}Q$:
\begin{equation}
    \left\{\left(e_i + r^{-2}e_n^*,e_n\right) : i\in \{1,\ldots,n-1\} \right\}\cup \left\{\left((1-r^{-2})^{-\h}e_n^*,(1-r^{-2})^{-\h}e_n)\right)\right\}.
\end{equation}

Using that any tangent vector to $T_{(x,q)}M$ can be written as
$\widetilde X:= \left(X+r^{-2}W^*+U^*,W+U\right)$ we compute
\begin{align*}
K_{\ga-r^{-2}}\left(\widetilde X,\left(e_i + r^{-2}e_n^*,e_n\right)\right) &= K_{\ga}\left(X+r^{-2}W^*+U^*,e_i+r^{-2}e_n^*\right),\\ i \in \{1,\ldots,n-1\}\\
 K_{\ga-r^{-2}}\left(\widetilde X,((1-r^{-2})^{-\h}e_n^*,(1-r^{-2})^{-\h}e_n))\right) &= K_{\ga}\left(X+r^{-2}W^*+U^*,(1-r^{-2})^{-\h}e_n^*\right)
\end{align*}
\begin{align*}  K_{\ga}\left(X+r^{-2}W^*+U^*,e_i+r^{-2}e_n^*\right) &= K_{\ga}(X,e_i) + r^{-4}K_{\ga}(X,e^*_n) + r^{-2}K_{\ga}(W^*,e_i) \\+ r^{-2}K_{\ga}(U^*,e_i)+2R_{\ga}(X,e_i,e^*_n,X) + 4r^{-2}R_{\ga}(X,e_i,e^*_n,W^*) \\+ 4R_{\ga}(X,e_i,e^*_n,U^*) + 2r^{-2}R_{\ga}(X,e_i,e_i,W^*) + 2R_{\ga}(X,e_i,e_i,U^*).\\
K_{\ga}\left(X+r^{-2}W^*+U^*,(1-r^{-2})^{-\h}e_n^*\right) &= (1-r^{-2})^{-1}K_{\ga}(X+r^{-2}W^*+U^*,e_n^*)\\
&= (1-r^{-2})^{-1}K_{\ga}(X,e_n^*) \\+ 2(1-r^{-2})^{-1}R_{\ga}(X,r^{-2}W^*+U^*,r^{-2}W^*+U^*,e_n^*).
\end{align*}
However,
\[R_{\ga}(X,r^{-2}W^*+U^*,r^{-2}W^*+U^*,e_n^*) = 0\]
\[R_{\ga}(X,e_i,e_i,U^*) = R_{\ga}(X,e_i,e_i,W^*) = R_{\ga}(X,e_i,e_n^*,X) = 0\]

Summing up
\begin{multline*}
    \mathrm{Ric}(\ga-r^{-2})(\widetilde X) = \left\{\dfrac{(r^2-1)(n-1)+r^6}{r^4(r^2-1)}\right\}K_{\ga}(X,e_n^*)   + \mathrm{Ric}^{\mathbf h}(X)\\ + r^{-2}\left\{\sum_{i=1}^{n-1}K_{\ga}(W^*,e_i) + K_{\ga}(U^*,e_i)\right\} + 4r^{-2}\sum_{i=1}^{n-1}\left\{R_{\ga}(X,e_i,e^*_n,W^*) + r^2R_{\ga}(X,e_i,e_n^*,U^*)\right\}.
\end{multline*}

Since $K_{\ga}(X,e_n^*) = |A^*_Xe_n^*|_{\ga}^2,~K_{\ga}(W^*,e_i) = |A^*_{e_i}W^*|_{\ga}^2,~K_{\ga}(U^*,e_i) = |A^*_{e_i}U^*|_{\ga}^2$
and
\[R_{\ga}(X,e_i,e_n^*,W^*) +r^2R_{\ga}(X,e_i,e_n^*,U^*) = \ga((\nabla^{\mathbf{v}}_{e^*_n}A)_Xe_i+A_{e_i}(A^*_Xe_n),W+r^2U^*)\]
we conclude
\begin{multline}\label{eq:Riccigeral}
     \mathrm{Ric}(\ga-r^{-2})(\widetilde X) = \left\{\dfrac{(r^2-1)(n-1)+r^6}{r^4(r^2-1)}\right\}|A_Xe_n^*|_{\ga}^2+ \mathrm{Ric}^{\mathbf h}(X)\\ + r^{-2}\left\{\sum_{i=1}^{n-1}|A^*_{e_i}W^*|_{\ga}^2 + |A^*_{e_i}U^*|_{\ga}^2\right\} + 4r^{-2}\sum_{i=1}^{n-1}\ga((\nabla^{\mathbf{v}}_{e^*_n}A)_Xe_i+A_{e_i}(A^*_Xe_n),W+r^2U^*).
\end{multline}

Finally, we claim
that
\begin{equation*}
\ga((\nabla^{\mathbf{v}}_{e^*_n}A)_Xe_i+A_{e_i}(A^*_Xe_n),W+r^2U^*) = \ga\left(A^*_{e_i}e^*_n,A^*_{X}(W^*+r^{-2}U^*)\right)
\end{equation*}
what shall conclude the proof since
\[\ga\left(A^*_{e_i}e^*_n,A^*_{X}(W^*+r^{-2}U^*)\right) = \ga\left(e_n^*,A_{e_i}A^*_X(W^*+r^{-2}U^*\right) = |A_{e_i}A^*_X(W^*+r^{-2}U^*)|_{\ga}\geq 0.\]

To verify the claim, it suffices to observe that (for the computation, we extend either $X$ or $V$ to geodesics and use equations \eqref{eq:dualnaintro}, \eqref{eq:codualnaintro}).
\begin{align*}
    (\nabla^{\mathbf v}_{e^*_n}A)_Xe_i &= \nabla_{e^*_n}(A_Xe_i) - A_{\nabla^{\mathbf h}_{e^*_n}X}e_i - A_X(\nabla^{\mathbf h}_{e^*_n}e_i)\\
    &= \nabla_{e^*_n}(A_Xe_i) +A_{A^*_Xe^*_n}e_i + A_X(A^*_{e_i}e^*_n)\\
    &= \nabla_{e^*_n}(A_Xe_i) - A_{e_i}(A^*_Xe^*_n) + A_X(A^*_{e_i}e^*_n)\\
    &= - A_{e_i}(A^*_Xe^*_n) + A_X(A^*_{e_i}e^*_n).
\end{align*}
So
\begin{align*}
   \ga((\nabla^{\mathbf{v}}_{e^*_n}A)_Xe_i+A_{e_i}(A^*_Xe_n),W+r^2U^*) &= \ga(A_X(A^*_{e_i}e^*_n),r^{-2}W^*+U^*).
   \end{align*}  
The almost semi-free case is a simple application of Corollary C item (a) in \cite{cavenaghiesilva} since the orbits in the complement of the principal stratum are isolated. \qedhere
\end{proof}

\section*{Acknowledgements}
The authors gladly acknowledge the anonymous referees for the careful read, leading to useful suggestions that improved the quality of the paper. The São Paulo Research Foundation (FAPESP) supports L. F. C.  grants 2022/09603-9, 2023/14316-1 and partially supports L. G. grants 2023/13131-8, 2021/04065-6.

	\bibliographystyle{alpha}
	
	\bibliography{main}

\end{document}